\documentclass[12pt,a4paper,reqno]{amsart}
\usepackage[utf8]{inputenc}
\usepackage[T1]{fontenc}
\usepackage{amsmath}
\usepackage{amsthm}
\usepackage{amssymb}
\usepackage[abbrev]{amsrefs}
\usepackage{mathrsfs}
\usepackage[dvipsnames]{xcolor}
\usepackage{bm}
\usepackage{enumitem}
\usepackage{hyperref}
\usepackage{bbm}

\makeatletter
\newcommand{\leqnos}{\tagsleft@true\let\veqno\@@leqno}
\newcommand{\reqnos}{\tagsleft@false\let\veqno\@@eqno}
\reqnos
\makeatother

\AtBeginDocument{\def\MR#1{}}

\newtheorem{thm}{}[section]
\newtheorem{theorem}[thm]{Theorem}
\newtheorem{corollary}[thm]{Corollary}
\newtheorem{lemma}[thm]{Lemma}
\newtheorem{proposition}[thm]{Proposition}

\theoremstyle{definition}
\newtheorem{definition}[thm]{Definition}
\theoremstyle{remark}
\newtheorem{remark}[thm]{Remark}

\newtheorem{example}[thm]{Example}

\newtheorem{clm}{}
\theoremstyle{remark}
\newtheorem{claim}[clm]{Claim}

\numberwithin{equation}{section}
\allowdisplaybreaks

\newcommand{\DD}{\ensuremath{\mathcal{D}}}
\newcommand{\KK}{\ensuremath{\mathcal{K}}}
\newcommand{\CC}{\ensuremath{\mathbb{C}}}
\newcommand{\FF}{\ensuremath{\mathbb{F}}}
\newcommand{\RR}{\ensuremath{\mathbb{R}}}
\newcommand{\NN}{\ensuremath{\mathbb{N}}}
\newcommand{\xx}{\ensuremath{\bm{x}}}
\newcommand{\yy}{\ensuremath{\bm{y}}}
\newcommand{\ee}{\ensuremath{\bm{e}}}
\newcommand{\zz}{\ensuremath{\bm{z}}}
\newcommand{\XX}{\ensuremath{\mathbb{X}}}
\newcommand{\XB}{\ensuremath{\mathcal{X}}}
\newcommand{\YY}{\ensuremath{\mathbb{Y}}}
\newcommand{\YB}{\ensuremath{\mathcal{Y}}}
\newcommand{\ZB}{\ensuremath{\mathcal{Z}}}
\newcommand{\EE}{\ensuremath{\mathcal{E}}}
\newcommand{\Ind}{\ensuremath{\mathbbm{1}}}
\newcommand{\plus}{\ensuremath{\bm{o}}}
\newcommand{\minus}{\ensuremath{\bm{e}}}

\newcommand{\gc}{\ensuremath{\overline{\bm{g}}}}
\newcommand{\g}{\ensuremath{\bm{g}}}
\newcommand{\h}{\ensuremath{\bm{h}}}
\newcommand{\UU}{\ensuremath{\mathcal{U}}}
\newcommand{\BB}{\ensuremath{\mathcal{B}}}
\newcommand{\co}{\ensuremath{\mathtt{c}}}
\newcommand{\uu}{\ensuremath{\bm{u}}}
\newcommand{\ww}{\ensuremath{\bm{w}}}

\newcommand{\n}{\ensuremath{\mathbf{n}}}
\newcommand{\Co}{\ensuremath{\mathbf{C}}}
\newcommand{\unc}{\ensuremath{\bm{k}}}
\newcommand{\BMO}{\ensuremath{\mathrm{BMO}}}
\newcommand{\udf}{\ensuremath{\bm{\varphi_u}}}
\newcommand{\ldf}{\ensuremath{\bm{\varphi_l}}}
\newcommand{\pb}{\ensuremath{\varphi_u}}
\DeclareMathOperator{\sgn}{sign}
\DeclareMathOperator{\spn}{span}
\DeclareMathOperator{\supp}{supp}



\begin{document}

\title[Bidemocratic bases]{Bidemocratic bases and their connections with other greedy-type bases}

\author[Albiac]{Fernando Albiac}
\address{Department of Mathematics, Statistics, and Computer Sciencies--InaMat2\\
Universidad P\'ublica de Navarra\\
Campus de Arrosad\'{i}a\\
Pamplona\\
31006 Spain}
\email{fernando.albiac@unavarra.es}

\author[Ansorena]{Jos\'e L. Ansorena}
\address{Department of Mathematics and Computer Sciences\\
Universidad de La Rioja\\
Logro\~no\\
26004 Spain}
\email{joseluis.ansorena@unirioja.es}

\author[Berasategui]{Miguel Berasategui}
\address{Miguel Berasategui\\
IMAS - UBA - CONICET - Pab I, Facultad de Ciencias Exactas y Naturales\\
Universidad de Buenos Aires\\
(1428), Buenos Aires, Argentina}
\email{mberasategui@dm.uba.ar}

\author[Bern\'a]{Pablo M. Bern\'a}
\address{Pablo M. Bern\'a\\
Departamento de Matem\'atica Aplicada y Estad\'istica, Facultad de Ciencias Econ\'omicas y Empresariales, Universidad San Pablo-CEU, CEU Universities\\
Madrid, 28003 Spain.}
\email{pablo.bernalarrosa@ceu.es}

\author[Lassalle]{Silvia Lassalle}
\address{Silvia Lassalle\\
Departamento de Matem\'atica\\
Universidad de San Andr\'es, Vito Duma 284\\
(1644) Victoria, Buenos Aires, Argentina and\\
IMAS - CONICET}
\email{slassalle@udesa.edu.ar}

\begin{abstract}
In nonlinear greedy approximation theory, bidemocratic bases have traditionally played the role of dualizing democratic, greedy, quasi-greedy, or almost greedy bases. In this article we shift the viewpoint and study them for their own sake, just as we would with any other kind of greedy-type bases. In particular we show that bidemocratic bases need not be quasi-greedy, despite the fact that they retain a strong unconditionality flavor which brings them very close to being quasi-greedy. Our constructive approach gives that for each $1<p<\infty$ the space $\ell_p$ has a bidemocratic basis which is not quasi-greedy. We also present a novel method for constructing conditional quasi-greedy bases which are bidemocratic, and provide a characterization of bidemocratic bases in terms of the new concepts of truncation
quasi-greediness and partially democratic bases.
\end{abstract}



\thanks{F. Albiac acknowledges the support of the Spanish Ministry for Science and Innovation under Grant PID2019-107701GB-I00 for \emph{Operators, lattices, and structure of Banach spaces}. F. Albiac and J.~L. Ansorena acknowledge the support of the Spanish Ministry for Science, Innovation, and Universities under Grant PGC2018-095366-B-I00 for \emph{An\'alisis Vectorial, Multilineal y Aproximaci\'on.} M. Berasategui and S. Lassalle were supported by ANPCyT PICT-2018-04104. P.~M. Bern\'a by Grants PID2019-105599GB-I00 (Agencia Estatal de Investigaci\'on, Spain) and 20906/PI/18 from Fundaci\'on S\'eneca (Regi\'on de Murcia, Spain). S. Lassalle was also supported in part by CONICET PIP 0483 and PAI UdeSA 2020-2021. F. Albiac, J.~L. Ansorena and P.~M. Bern\'a would like to thank the Erwin Schr\"odinger International Institute for Mathematics and Physics, Vienna, for support and hospitality during the programme \emph{Applied Functional Analysis and High-Dimensional Approximation}, held in the Spring of 2021, where work on this paper was undertaken.}

\maketitle

\section{Introduction and background}\noindent
Let $\XX$ be an infinite-dimensional separable Banach space (or, more generally, a quasi-Banach space) over the real or complex field $\FF$. Throughout this paper, unless otherwise stated, by a \emph{basis} of $\XX$ we mean a norm-bounded sequence $\XB=(\xx_n)_{n=1}^\infty$ that generates the entire space, in the sense that
\[
\overline{\spn}(\xx_n \colon n\in\NN)=\XX,
\]
and for which there is a (unique) norm-bounded sequence $\XB^{\ast}=(\xx_{n}^{\ast})_{n=1}^\infty$ in the dual space $\XX^{\ast}$ such that $(\xx_{n}, \xx_{n}^{\ast})_{n=1}^{\infty}$ is a biorthogonal system.
We will refer to the basic sequence $\XB^{\ast}$ as to the \emph{dual basis} of $\XB$.

We recall that the basis $\XB=(\xx_n)_{n=1}^\infty$ is called \emph{democratic} if there is a constant $\Delta$ such that
\[
\left\Vert \sum_{k\in A}\xx_k \right\Vert\le \Delta \left\Vert \sum_{k\in B}\xx_k \right\Vert
\]
whenever $A$ and $B$ are finite subsets of $\NN$ with $|A|\le |B|$. The \emph{fundamental function} $\varphi\colon\NN\to[0,\infty)$ of $\XB$ is then defined by
\[
\varphi(m)=\sup\limits_{|A|\le m}\left\Vert \sum_{k\in A}\xx_k \right\Vert,\quad m\in\NN,
\]
while the \emph{dual fundamental function} of $\XB$ is just the fundamental function of its dual basis, i.e.,
\[
\varphi^{\ast}(m)=\sup\limits_{|A|\le m}\left\Vert \sum_{k\in A}\xx_k^{\ast} \right\Vert, \quad m\in \NN.
\]
In general it is not true that if a basis $\XB=(\xx_n)_{n=1}^\infty$ is democratic, then its dual basis $\XB^{\ast}$ is democratic as well. For instance, the $L_1$-normalized Haar system is an unconditional democratic basis of the dyadic Hardy space $H_1$ \cites{Woj1982,Woj2000}, but the $L_\infty$-normalized Haar system is not democratic in the dyadic $\BMO$-space \cite{Oswald2001}. In order to understand better how certain greedy-like properties dualize, Dilworth et al.\ introduced in \cite{DKKT2003} a strengthen form of democracy. Notice that the elementary computation
\[
m=\left(\sum_{k\in A} \xx_k^{\ast}\right)\left(\sum_{k\in A} \xx_k\right)
\le \left\Vert\sum_{k\in A} \xx_k^{\ast}\right\Vert \left\Vert\sum_{k\in A} \xx_k\right\Vert
\;\text{if}\;
|A|=m,
\]
yields the estimate
\[
m\le \varphi(m)\,\varphi^{\ast}(m),\quad m\in\NN.
\]
A basis $\XB=(\xx_n)_{n=1}^\infty$ is then said to be \emph{bidemocratic} if the reverse inequality is fulfilled up to a constant, i.e., $\XB$ is bidemocratic with constant $\Delta_{b}$ ($\Delta_{b}$-bidemocratic for short)
if
\[
\varphi(m)\, \varphi^{\ast}(m)\le \Delta_{b}\, m, \quad m\in \NN.
\]
Amongst other relevant results relative to this kind of bases in Banach spaces, Dilworth et al.\ showed in \cite{DKKT2003} that being quasi-greedy passes to dual bases under the assumption of bidemocracy (see \cite{DKKT2003}*{Theorem 5.4}). Since the dual basis of a bidemocratic basis is democratic, it follows that the corresponding result also holds for almost greedy and greedy bases. That is, if a bidemocratic basis is almost greedy (respectively, greedy), then so is its dual basis.

Despite the instrumental role played by bidemocratic bases as a key that permits dualizing some greedy-type properties, it is our contention in this paper that these bases are of interest by themselves and that they deserve to be studied as any other kind of greedy-like bases. For instance, the unconditionality constants of bidemocratic bases have been estimated (see Theorem~\ref{thm:BidCond} below), which sheds some information onto the performance of the greedy algorithm when it is implemented specifically for these bases.

To undertake our task we must first place bidemocratic bases in the map by relating them with other types of bases that are relevant in the theory. In this respect the most important open question is whether bidemocratic bases are quasi-greedy. This problem is motivated by recent results that show that bidemocratic bases have uniform boundedness properties of certain (nonlinear) truncation operators that make them very close to quasi-greedy bases (see \cite{AABW2021}*{Proposition 5.7}). In our language, bidemocratic bases are truncation quasi-greedy. In Section~\ref{sect:BDNonQG} we will solve this question in the negative by proving that bidemocracy is not in general strong enough to ensure quasi-greediness and show that for $1<p<\infty$ the space $\ell_p$ has a bidemocratic basis which is not quasi-greedy.

Before that, we will look for sufficient conditions for a basis to be bidemocratic. Here one must take into account that if $\XB$ is bidemocratic then both $\XB$ and $\XB^{\ast}$ are democratic but the converse fails. The only positive result we find in the literature in the reverse direction is the aforementioned Theorem 5.4 from \cite{DKKT2003}, which tells us that if $\XB$ and $\XB^{\ast}$ are quasi-greedy and democratic then $\XB$ is bidemocratic. In Section~\ref{sect:truncation quasi-greedy} we extend this result by relaxing the conditions on the bases $\XB$ and $\XB^{\ast}$ while still attaining the bidemocracy of $\XB$.

Turning to quasi-greedy bases, it is natural and consistent with our discussion in this paper, to further the study of conditional quasi-greedy bases by looking for conditional bidemocratic quasi-greedy bases, i.e., conditional almost greedy bases whose dual bases are also almost greedy. The previous methods for building conditional almost greedy bases in Banach spaces yield either bases whose fundamental function coincides with the fundamental function of the canonical basis of $\ell_1$, or bases whose fundamental function increases steadily enough (formally, bases that have the upper regularity property and the lower regularity property). In the former case, the bases are not bidemocratic unless they are equivalent to the canonical $\ell_1$-basis; in the latter, the bases are always bidemocratic by \cite{DKKT2003}*{Proposition 4.4}. The existence of conditional bidemocratic quasi-greedy bases which do not have the upper regularity property seems to be an unexplored area. In Section~\ref{sect:NM} we contribute to this topic by developing a new method for building bidemocratic, conditional, quasi-greedy bases with arbitrary fundamental functions.

Throughout this paper we will use standard notation and terminology from Banach spaces and greedy approximation theory, as can be found, e.g., in \cite{AlbiacKalton2016}. We also refer the reader to the recent article \cite{AABW2021} for other more especialized notation. We next single out however the most heavily used terminology.

For broader applicability, whenever it is possible we will establish our results in the setting of quasi-Banach spaces. Let us recall that a \emph{quasi-Banach space} is a vector space $\XX$ over the real or complex field $\FF$ equipped with a \emph{quasi-norm}, i.e., a map $\|\cdot\|\colon \XX\to [0,\infty)$ that satisfies all the usual properties of a norm with the exception of the triangle law, which is replaced with the condition
\begin{equation}\label{defquasinorm}
\|f+g\|\leq \kappa( \| f\| + \|g\|),\quad f,g\in \XX,
\end{equation}
for some $\kappa\ge 1$ independent of $f$ and $g$, and moreover $(\XX,\|\cdot\|)$ is complete.
The \emph{modulus of concavity} of the quasi-norm is the smallest constant $\kappa\ge 1$ in \eqref{defquasinorm}. Given $0<p\le 1$, a \emph{$p$-Banach space} will be a quasi-Banach space whose quasi-norm is $p$-subadditive, i.e.,
\[
\Vert f+g\Vert^p \le \Vert f\Vert^p +\Vert g \Vert^p, \quad f,g\in\XX.
\]

Some authors have studied the Thresholding Greedy Algorithm, or TGA for short, for more demanding types of bases that we will bring into play on occasion. A sequence $\XB=(\xx_n)_{n=1}^\infty$ of $\XX$ is said to be a \emph{Schauder basis} if for every $f\in\XX$ there is a unique sequence $(a_n)_{n=1}^\infty$ in $\FF$ such that $f= \sum_{n=1}^{\infty} a_n\, \xx_{n}$, where the convergence of the series is understood in the topology induced by the quasi-norm. If $\XB$ is a Schauder basis we define the biorthogonal functionals associated to $\XB$ by $\xx_k^*(f)=a_k$ for all $f=\sum_{n=1}^{\infty} a_n \, \xx_{n}\in\XX$ and $k\in\NN$. The \emph{partial-sum projections} $S_{m}\colon \XX\to \XX$ with respect to the Schauder basis $\XB$, given by
\[
f\mapsto S_{m}(f)= \sum_{n=1}^{m} \xx_n^*(f)\, \xx_{n}, \quad f\in\XX,\, m\in\NN,
\]
are uniformly bounded, whence we infer that $\sup_n \Vert \xx_n\Vert \, \Vert \xx_n^*\Vert<\infty$. Hence, if a Schauder basis $\XB$ is semi-normalized, i.e.,
\[
0<\inf_n \Vert \xx_n\Vert\le \sup_n \Vert \xx_n\Vert<\infty,
\]
then $(\xx_n^*)_{n=1}^\infty$ is norm-bounded and so $\XB$ is a basis in the sense of this paper.
If $\XB=(\xx_n)_{n=1}^\infty$ is a Schauder basis, then the \emph{coefficient transform}
\[
f\mapsto (\xx_n^{\ast}(f))_{n=1}^\infty, \quad f\in\XX,
\]
is one-to-one, that is, the basis $\XB$ is \emph{total}. In the case when $\Vert S_m\Vert \le 1$ for all $m\in\NN$ the Schauder basis $\XB$ is said to be \emph{monotone}.

Given $A\subseteq \NN$, we will use $\EE_A$ to denote the set consisting of all families $(\varepsilon_n)_{n\in A}$ in $\FF$ with $|\varepsilon_n|=1$ for all $n\in A$. Given a basis $\XB=(\xx_n)_{n=1}^\infty$ of $\XX$, a finite set $A\subseteq\NN$ and $\varepsilon=(\varepsilon_n)_{n\in A}\in\EE_A$ it is by now customary to use
\[
\textstyle
\Ind_{\varepsilon,A}[\XB,\XX]=\sum_{n\in A} \varepsilon_n\, \xx_n
\;(\text{resp.,}\;
\Ind^{\ast}_{\varepsilon,A}[\XB,\XX]=\sum_{n\in A} \varepsilon_n \, \xx_n^{\ast}
).
\]
If the basis and the space are clear from context we simply put $\Ind_{\varepsilon,A}$ (resp., $\Ind^{\ast}_{\varepsilon,A}$), and if $\varepsilon_n=1$ for all $n\in A$ we put $\Ind_A$ (resp., $\Ind^{\ast}_{A}$). Associated with the fundamental function $\varphi$ of the basis are the \emph{upper super-democracy function} of $\XB$,
\[
\udf(m)=\udf[\XB,\XX](m)=\sup\left\lbrace \left\Vert \Ind_{\varepsilon,A} \right\Vert \colon |A|\le m,\, \varepsilon\in\EE_A \right\rbrace, \quad m\in\NN.
\]
and the \emph{lower super-democracy function} of $\XB$,
\[
\ldf(m)=\ldf[\XB,\XX](m)=\inf\left\lbrace \left\Vert \Ind_{\varepsilon,A} \right\Vert \colon |A|\ge m,\, \varepsilon\in\EE_A \right\rbrace, \quad m\in\NN.
\]
The growth of $\udf$ is of the same order as $\varphi$ (see \cite{AABW2021}*{inequality (8.3)}), and so the
basis $\XB$ is bidemocratic if and only if
\begin{equation*}
\sup_{m\in\NN} \frac{1}{m} \udf[\XB,\XX](m) \udf[\XB^{\ast},\XX^{\ast}](m) <\infty
\end{equation*}
(see \cite{AABW2021}*{Lemma 5.5}).

The symbol $\alpha_j\lesssim \beta_j$ for $j\in J$ means that there is a positive constant $C$ such that the families of nonnegative real numbers $(\alpha_j)_{j\in J}$ and $(\beta_j)_{j\in J}$ are related by the inequality $\alpha_j\le C\beta_j$ for all $j\in J$. If $\alpha_j\lesssim \beta_j$ and $\beta_j\lesssim \alpha_j$ for $j\in J$ we say $(\alpha_j)_{j\in J}$ are $(\beta_j)_{j\in J}$ are equivalent, and write $\alpha_j\approx \beta_j$ for $j\in J$.

We finally recall that two bases $(\xx_n)_{n=1}^\infty$ and $(\yy_n)_{n=1}^\infty$ of quasi-Banach spaces $\XX$ and $\YY$ are said to be equivalent if there is an isomorphism $T$ from $\XX$ onto $\YY$ with $T(\xx_n)=\yy_n$ for all $n\in\NN$.

\section{From truncation quasi-greedy to bidemocratic bases }\label{sect:truncation quasi-greedy}\noindent
Let $\XB=(\xx_{n})_{n=1}^{\infty}$ be a semi-normalized basis for a quasi-Banach space $\XX$
with dual basis $(\xx_{n}^*)_{n=1}^{\infty}$.
For each $f\in \XX$ and each $B\subseteq\NN$ finite, put
\[
\UU(f,B) = \min_{n\in B} |\xx_n^{\ast}(f)| \sum_{n\in B} \sgn (\xx_n^{\ast}(f)) \, \xx_n.
\]
Given $m\in\NN\cup\{0\}$, the $m$\emph{th-restricted truncation operator} $\UU_m\colon \XX \to \XX$ is defined as
\[
\UU_m(f)=\UU(f,A_m(f)), \quad f\in\XX,
\]
where $A=A_m(f)\subseteq\NN$ is a \emph{greedy set} of $f$ of cardinality $m$, i.e., $|\xx_{n}^{\ast}(f)|\ge| \xx_{k}^{\ast}(f)|$ whenever $n\in A$ and $k\not\in A$. The set $A$ depends on $f$ and $m$, and may not be unique; if this happens we take any such set. We put
\begin{equation*}
\Lambda_u=\Lambda_u[\XB,\XX]=\sup\{ \Vert \UU(f,B)\Vert \colon B
\;\text{greedy set of}\;
f, \, \Vert f\Vert \le 1\}.
\end{equation*}
If the quasi-norm is continuous, applying a perturbation technique yields
\[
\Lambda_u=\sup_m \Vert \UU_m\Vert.
\]
Thus, the basis $\XB$ is said to be \emph{truncation quasi-greedy} if $(\UU_m)_{m=1}^\infty$ is a uniformly bounded family of (nonlinear) operators, or equivalently, if and only if $\Lambda_u<\infty$. In this case we will refer to $\Lambda_u$ as the \emph{truncation quasi-greedy constant} of the basis.

Quasi-greedy bases are truncation quasi-greedy (see \cite{DKKT2003}*{Lemma 2.2} and \cite{AABW2021}*{Theorem 4.13}), but the converse does not hold in general. The first case in point appeared in the proof of \cite{BBG2017}*{Proposition 5.6}, where the authors constructed a basis that dominates the unit vector system of $\ell_{1,\infty}$, hence it is truncation quasi-greedy by \cite{AABW2021}*{Proposition 9.4}, but it is not quasi-greedy. In spite of that, truncation quasi-greedy bases still enjoy most of the nice unconditionality-like properties of quasi-greedy bases. For instance, they are quasi-greedy for large coefficients (QGLC for short), suppression unconditional for constant coefficients (SUCC for short), and lattice partially unconditional (LPU for short). See \cite{AABW2021}*{Sections 3 and 4} for the precise definitions and the proofs of these relations.

In turn, if $\XB$ is bidemocratic then both $\XB$ and its dual basis $\XB^{\ast}$ are truncation quasi-greedy (\cite{AABW2021}*{Proposition 5.7}). In this section we study the converse implication, i.e., we want to know which additional conditions make a truncation quasi-greedy basis bidemocratic. A good starting point is the following result, which uses the upper regularity property (URP for short) and which is valid only for Banach spaces. Following \cite{DKKT2003} we shall say that a basis has the URP if there is an integer $b\ge 3$ so that its fundamental function $\varphi$ satisfies
\begin{equation}\label{URPdef}
2\varphi(b m)\le {b} \varphi(m),\quad m\in\NN.
\end{equation}

\begin{theorem}[see \cite{AABW2021}*{Lemma 9.8 and Proposition 10.17(iii)}]
Let $\XB$ be a basis of a Banach space $\XX$. Suppose that $\XB$ is democratic, truncation quasi-greedy, and has the URP. Then $\XB$ is bidemocratic (and so $\XB^{\ast}$ is truncation quasi-greedy too).
\end{theorem}

Can we do any better? Dilworth et al.\ characterized bidemocratic bases as those quasi-greedy bases which fulfill an additional condition, weaker than democracy, which they named conservative (\cite{DKKT2003}*{Theorem 5.4}). Recall that a basis is said to be \emph{conservative} if there is a constant $C$ such that $\Vert \Ind_{A} \Vert \le C \Vert \Ind_{B} \Vert$ whenever $|A|\le |B|$ and $\max (A) \le \min (B)$. Our objection to this concept is that it is not preserved under rearrangements of the basis. Thus, since the greedy algorithm is ``reordering invariant'' (i.e., if $\pi$ is a permutation of $\NN$, the greedy algorighm with respect to the bases $(\xx_n)_{n=1}^\infty$ and $(\xx_{\pi(n)})_{n=1}^\infty$ is the same) when working with conservative bases we are bringing an outer element into the theory. This is the reason why we establish our characterization of bidemocratic bases below in terms of a reordering invariant new class of bases which is more general than the class of conservative bases and which we next define.

\begin{definition}
We say that a basis is \emph{partially democratic} if there is a constant $C$ such that for each $D\subseteq\NN$ finite there is $D\subseteq E\subseteq\NN$ finite such that $\Vert \Ind_A\Vert \le C \Vert \Ind_B\Vert$ whenever $A\subseteq D$ and $B\subseteq \NN\setminus E$ satisfy $|A|\le |B|$.
\end{definition}

The following lemma is well-known.
\begin{lemma}[See \cite{AABW2021}*{Proposition 4.16} or \cite{AAW2021b}*{Lemma 5.2}]\label{lem:truncation quasi-greedyQU}
Suppose that $\XB=(\xx_n)_{n=1}^\infty$ is a truncation quasi-greedy basis of a quasi-Banach space $\XX$. Then there is a constant $C$ depending on the modulus of concavity of $\XX$ and the truncation quasi-greedy constant of $\XB$ such that
\[
\left\Vert \sum_{n\in A} a_n\, \xx_n\right\Vert \le C \Vert f\Vert
\]
for all $f\in \XX$, and $A\subseteq\NN$ such that $\max_{n\in A}|a_n|\le \min_{n\in A} |\xx_n^{\ast}(f)|$.
\end{lemma}

\begin{theorem}\label{thm:PDtruncation quasi-greedy}
Let $\XB$ be a basis of a Banach space $\XX$. Suppose that both $\XB$ and $\XB^{\ast}$ are truncation quasi-greedy and partially democratic. Then $\XB$ is bidemocratic.
\end{theorem}

\begin{proof}
We will customize the proof of \cite{DKKT2003}*{Theorem 5.4} to suit our more general statement. By Lemma~\ref{lem:truncation quasi-greedyQU} there is a constant $\Lambda$ such that
\begin{equation}
\Vert \Ind_{\varepsilon,A} \Vert\le \Lambda \Vert \Ind_{\varepsilon,A} + f\Vert \label{anotherone}
\end{equation}
for every $A\subseteq \NN$ finite, every $\varepsilon\in\EE_A$, and every $f\in\XX$ with $\supp(f)\cap A=\emptyset$.
Applying the Hahn--Banach theorem to the equivalence class of $\Ind_{\varepsilon,A}$ in the quotient space $\XX/\overline{\spn}(\xx_n\colon n\notin A)$ yields $f^*\in\spn(\xx_n^* \colon n\in A)$ with $\|f^{\ast}\|=1$ such that
\[
{\|\Ind_{\varepsilon, A}\|}\le {\Lambda} |f^{\ast}(\Ind_{\varepsilon,A})|.
\]

Set $\udf=\udf[\XB,\XX]$ and $\udf^{\ast}=\udf[\XB^{\ast},\XX^{\ast}]$. Let $\Delta_d$ and $\Delta_d^{\ast}$ be the partial democracy constants of $\XB$ and $\XB^{\ast}$ respectively, and let $\Lambda_u^{\ast}$ be the truncation quasi-greedy constant of $\XB^{\ast}$. Given $m\in \NN$, fix $0<\epsilon<1$ and choose sets $B_1, B_2$ and signs $\varepsilon\in \EE_{B_1}$, $\varepsilon'\in \EE_{B_2}$ so that $|B_1|\le m$, $|B_2|\le m$,
\begin{align}
\|\Ind_{\varepsilon, B_1}\|\ge (1-\epsilon) \udf(m)
\;\text{and}\;
\|\Ind_{\varepsilon', B_2}^{\ast}\|\ge (1-\epsilon)\udf^{\ast}(m). \label{two}
\end{align}
Use partial democracy to pick $D\subseteq\NN$ disjoint with $B_1\cup B_2$ such that $|D|=2m$, $\Vert \Ind_B\Vert \le \Co \Vert \Ind_A\Vert$, and $\Vert \Ind_B^{\ast}\Vert \le \Co^{\ast} \Vert \Ind_A^{\ast}\Vert$ whenever $B\subseteq B_1\cup B_2$ and $A\subseteq D$ satisfy $|B|\le |A|$.

It follows from \eqref{two} and partial democracy that, for every $A\subseteq D$ with $|A|\ge m$,
\begin{align}
(1-\epsilon)\udf(m)\le& \Co\|\Ind_{A}\|,
\;\text{and}\;
(1-\epsilon)\udf^{\ast}(m)\le \Co^{\ast}\|\Ind_{A}^{\ast}\|,\label{one3}
\end{align}
where $\Co= 2\lambda \Delta_d$ and $\Co^{\ast}=2 \lambda \Delta_d^{\ast}$ with $\lambda=1$ if $\FF=\RR$ or $2$ if $\FF=\CC$.
For such subsets $A$ of $\NN$ the set
\[
\KK_A=\left\{f^{\ast}\in
\spn(\xx_n^{\ast} \colon n\in A) \colon \|f^{\ast}\|\le 1, \; f^{\ast}(\Ind_{A})\ge \frac{ (1-\epsilon)\pb(m)}{\Co\Lambda}\right\}
\]
is convex and nonempty. Note that $\KK_A$ increases with $A$, and that
\begin{equation}\label{eq:TrivialEst}
\sum_{n\in A} |f^{\ast}(\xx_n)|
= f^{\ast}\left( \Ind_{\overline{\varepsilon(f^{\ast})},A}\right)
\le \Vert f^{\ast}\Vert \, \left\Vert \Ind_{\overline{\varepsilon(f^{\ast})},A}\right\Vert
\le \pb(|A|), \; f^{\ast}\in \KK_A.
\end{equation}

Pick $f^{\ast}\in\KK_D$ that minimizes $\sum_{n\in D}|f^{\ast}(\xx_n)|^2$. The geometric properties of minimizing vectors on convex subsets of Hilbert spaces yield
\begin{equation}\label{eq:GeoH}
\sum_{n\in D}|f^{\ast}(\xx_n)|^2\le \Re\left( \sum_{n\in D} f^{\ast}(\xx_n) g^{\ast}(\xx_n)\right), \quad g^{\ast}\in \KK_D.
\end{equation}
Let $E$ be a greedy set of $f^{\ast}$ with $|E|=m$, and put $A=D\setminus E$. Using that $\XB^{\ast}$ is truncation quasi-greedy we obtain
\begin{equation}\label{eq:truncation quasi-greedyD}
\min_{n\in E}|f^{\ast}(\xx_n)|\, \|\Ind_{E}^{\ast}\|\le \Lambda_u^{\ast}\|f^{\ast}\|\le \Lambda_u^{\ast}.
\end{equation}
Pick $g^{\ast}\in \KK_A$. By \eqref{eq:GeoH}, \eqref{eq:TrivialEst}, \eqref{eq:truncation quasi-greedyD} and \eqref{one3},
\begin{align*}
\sum_{n\in D}|f^{\ast}(\xx_n)|^2&\le \sum_{n\in A}|f^{\ast}(\xx_n)||g^{\ast}(\xx_n)|\\
&\le \min_{n\in E}|f^{\ast}(\xx_n)| \sum_{n\in A} |g^{\ast}(\xx_n)|\\
&\le\frac{\Lambda_u^{\ast}}{\Vert \Ind_E^{\ast}\Vert} \udf(m) \\
&\le\frac{\Lambda_u^{\ast}\Co^{\ast}}{(1-\epsilon)\udf^{\ast}(m)}\pb(m).
\end{align*}
Hence, by the Cauchy--Bunyakovsky--Schwarz inequality,
\begin{align*}
(1-\epsilon)^{2}(\pb(m))^2&\le \Co^2\Lambda^2 | f^{\ast}(\Ind_D)|^2\\
&\le \Co^2\Lambda^2\left(\sum_{n\in D}|f^{\ast}(\xx_n)|\right)^2\\
&\le 2 \Co^2\Lambda^2 m \sum_{n\in D}|f^{\ast}(\xx_n)|^2\\
& \le 2m \frac{\Co^2\Co^{\ast} \Lambda^2\Lambda_u^{\ast}}{(1-\epsilon)} \frac{\udf(m)}{\udf^{\ast}(m)}.
\end{align*}
Since $\epsilon$ is arbitrary, we obtain
\[
\udf(m)\udf^{\ast}(m)\le 2 \Co^2\Co^{\ast}\Lambda^2\Lambda_u^{\ast}m,
\]
and so the basis is bidemocratic.
\end{proof}

\begin{corollary}
Let $\XB$ be a basis of a Banach space $\XX$. Suppose that both $\XB$ and $\XB^{\ast}$ are truncation quasi-greedy and conservative. Then $\XB$ is bidemocratic.
\end{corollary}

\begin{proof}
If follows readily from Theorem~\ref{thm:PDtruncation quasi-greedy} since conservative bases are partially democratic.
\end{proof}

\begin{remark}
Note that Theorem~\ref{thm:PDtruncation quasi-greedy} makes sense only for Banach spaces, i.e., it cannot be extended to nonlocally convex quasi-Banach spaces. Indeed, for $0<p<1$ the unit vector system of $\ell_p$ is a democratic unconditional basis whose dual basis is the unit vector system of $c_0$, which also is democratic; but the unit vector system of $\ell_p$ is not bidemocratic!
\end{remark}

\section{Existence of bidemocratic non-quasi-greedy bases}\label{sect:BDNonQG}\noindent
This section is geared towards proving the existence of bidemocratic bases which are not quasi-greedy. To that end, let us first set the minimum requirements on terminology we need for this section.

Suppose $\XB=(\xx_{n})_{n=1}^{\infty}$ is a democratic basis in a quasi-Banach space $\XX$. We shall say that $\XB$ has the \emph{lower regularity property} (LRP for short) if there is an integer $b\ge 2$ such
\begin{equation}\label{LRPdef}
2 \varphi(m) \le \varphi(bm), \quad m\in\NN.
\end{equation}
In a sense, the LRP is the dual property of the URP. Abusing the language we will say that a sequence has the URP (respectively, LRP), if its terms verify the condition \eqref{URPdef} (respectively, \eqref{LRPdef}). Note that $(\varphi(m))_{m=1}^\infty$ has the LRP if and only if $(m/\varphi(m))_{m=1}^\infty$ has the URP. If $(\varphi(m))_{m=1}^\infty$ has the LRP then there is $a>0$ and $C\ge 1$ such that
\begin{equation}\label{eq:LRP}
\frac{m^a}{n^a}\le C \frac{\varphi(m)}{\varphi(n)}, \quad n\le m.
\end{equation}
In the case when $\varphi$ is non-decreasing and the sequence $(\varphi(m)/m)_{m=1}^\infty$ is non-increasing, $\varphi$ has the LRP if and only if the weight $\ww=(w_n)_{n=1}^\infty$ defined by $w_n=\varphi(n)/n$ is a \emph{regular} weight, i.e., it satisfies the Dini condition
\[
\sup_{n} \frac{1}{n w_n} \sum_{k=1}^n w_k <\infty
\]
(see \cite{AABW2021}*{Lemma 9.8}), in which case
\begin{equation}\label{eq:LRPbis}
\sum_{n=1}^m \frac{\varphi(n)}{n} \approx \varphi(m), \quad m\in\NN.
\end{equation}
For instance, the power sequence $(m^{1/p})_{m=1}^{\infty}$ has the URP for $1<p<\infty$. In other words, the weight $\ww= (n^{-a})_{n=1}^{\infty}$ is regular for $0<a<1$.

We will need the following elementary lemma about the \emph{harmonic numbers}
\[
H_m=\sum_{n=1}^m \frac{1}{n}, \quad m\in\NN\cup\{0\}.
\]
\begin{lemma}\label{lem:JarDif}
For each $0<a<1$ there exists a constant $C(a)$ such that
\begin{equation*}
S(a,r,t):=\sum_{k=r+1}^t k^{-a}(k-r)^{a-1}\le C(a) (H_t-H_r), \quad t\ge 2r.
\end{equation*}
\end{lemma}
\begin{proof}
The inequality is trivial for $r=0$. So we assume that $r\ge 1$. If we define $f\colon [1,\infty) \to [0,\infty)$ by
$
f(u) = u^{-a} (u-1)^{a-1},
$
we have
\[
k^{-a}(k-r)^{a-1} \le x^{-a} (x-r)^{a-1}= \frac{1}{r} f\left( \frac{x}{r}\right), \quad k\in \NN, \; x\in [k-1,k].
\]
Hence,
\[
S(a,r,t)\le \int_{r}^t f\left( \frac{x}{r}\right) \frac{dx}{r}=\int_1^{t/r} f(u) \, du.
\]
Since $f$ is integrable on $[1,2]$ and $f(u) \lesssim 1/u$ for $u\in[2,\infty)$, there is a constant $C_1$ such that $S(a,r,t) \le C_1 \log(t/r)$. Taking into account that $H_t-H_r\ge (t-r)/t\ge 1/2$, and that there is a constant $C_2$ such that $ \log m\le H_m\le \log m+C_2$ for all $m\in\NN$ we are done.
\end{proof}

For further reference, we record an easy lemma that we will use several times. Note that it applies in particular to the harmonic series.

\begin{lemma}\label{lem:Jar}
Let $\sum_{n=1}^\infty c_n$ be a divergent series of nonnegative terms. Suppose that $\lim_n c_n=0$. Then, for every $m\in\NN\cup\{0\}$ and $0\le a<b$, there are $m\le r<s$ such that $a\le \sum_{n=r+1}^s c_n <b$.
\end{lemma}

We will also use the following well-known lemma. Note that it could be used to prove the divergence of the harmonic series.

\begin{lemma}[See \cite{Rudin1976}*{Exercise 11, p.\ 84}]\label{lem:AlsoDiverges}
Let $\sum_{n=1}^\infty c_n$ be a divergent series of nonnegative terms. Then the (smaller) series
\[
\sum_{n=1}^\infty \frac{c_n}{\sum_{k=1}^n c_k}
\]
also diverges.
\end{lemma}

Lorentz sequence spaces $d_{1,q}(\ww)$ play a relevant role in the qualitative study of greedy-like bases. Let $\ww=(w_n)_{n=1}^\infty$ be a weight (i.e., a sequence of nonnegative numbers with $w_1>0$) whose primitive weight $(s_m)_{m=1}^\infty$, defined by $s_m=\sum_{n=1}^m w_n$, is unbounded and \emph{doubling}, i.e.,
\[
\sup_m \frac{s_{2m}}{s_m} <\infty.
\]
Given $0<q\le \infty$, we will denote by $d_{1,q}(\ww)$ the quasi-Banach space of all $f\in c_0$ whose non-increasing rearrangement $(a_n)_{n=1}^\infty$ satisfies
\[
\Vert f\Vert_{d_{1,q}(\ww)}=\left( \sum_{n=1}^\infty a_n^q s_n^{q-1}w_n\right)^{1/q}<\infty,
\]
with the usual modification if $q=\infty$. For power weights this definition yields the classical Lorentz sequence spaces $\ell_{p,q}$. To be precise, if $\ww=(n^{1/p-1})_{n=1}^\infty$ for some $0<p<\infty$, then, up to an equivalent quasi-norm, $d_{1,q}(\ww) =\ell_{p,q}$ and if $(a_n)_{n=1}^\infty$ is the non-increasing rearrangement of $f\in c_0$,
\[
\Vert f \Vert_{\ell_{p,q}}=\left( \sum_{n=1}^\infty a_n^q n^{q/p-1}\right)^{1/q}.
\]
For a quick introduction to Lorentz sequence spaces, we refer the reader to \cite{AABW2021}*{Section 9.2}. Here we gather the properties of these spaces that are most pertinent for our purposes. Although it is customary to designate them after the weight $\ww$, it must be conceded that as a matter of fact they depend on its primitive weight $(s_m)_{m=1}^\infty$ rather than on $\ww$. That is, given weights $\ww=(w_n)_{n=1}^\infty$ and $\ww'=(w_n')_{n=1}^\infty$ with primitive weights $(s_m)_{m=1}^\infty$ and $(s_m')_{m=1}^\infty$, we have $d_{1,q}(\ww)=d_{1,q}(\ww')$ (up to an equivalent quasi-norm) if and only if $s_m\approx s_m'$ for $m\in\NN$. The fundamental function of the unit vector system of $d_{1,q}(\ww)$ is equivalent to $(s_m)_{m=1}^\infty$ thus, essentially, it does not depend on $q$. We have
\[
d_{1,p}(\ww) \subseteq d_{1,q}(\ww), \quad 0<p<q\le \infty.
\]
To show that this inclusion is actually strict we can, for instance, use the sequence
\[
H_m[\ww]=\sum_{n=1}^m \frac{w_n}{s_n}, \quad m\in\NN,
\]
and notice that $\lim_m H_m[\ww]=\infty$ by Lemma~\ref{lem:AlsoDiverges}, and
\begin{equation}\label{eq:NormLorentz}
\left\Vert \sum_{n=1}^m \frac{1}{s_n}\, \ee_n\right\Vert_{d_{1,q}(\ww)}=(H_m[\ww])^{1/q}, \quad m\in\NN,\; 0<q<\infty.
\end{equation}

\begin{lemma}\label{lem:LorentzLRP}
Let $0<q\le \infty$, and let $(s_m)_{m=1}^\infty$ be the primitive weight of a weight $\ww$. Suppose that $(s_m)_{m=1}^\infty$ has the LRP and that the weight $\ww'=(w_n')_{n=1}^\infty$ given by $w_n'=s_n/n$ is non-increasing. Then:
\begin{enumerate}[label=(\roman*), leftmargin=*, widest=iii]
\item\label{LorentzLRP:1} $d_{1,q}(\ww)=d_{1,q}(\ww')$;
\item\label{LorentzLRP:2} for $0\le r \le t<\infty$, $H_t[\ww']-H_r[\ww']\approx H_t-H_r$ and
\item\label{LorentzLRP:3}
$
A(r,t):=\left\Vert \sum_{n=r+1}^t s_n^{-1}\, \ee_n \right\Vert _{d_{1,q}(\ww)} \lesssim \max\{1, (H_t-H_r)^{1/q}\}.
$
\end{enumerate}
\end{lemma}

\begin{proof}
The first part follows from \eqref{eq:LRPbis}. Let $(s_m')_{m=1}^\infty$ be the primitive weight of $\ww'$. The equivalence \eqref{eq:LRPbis} also yields
\[
\frac{w_n'}{s_n'}\approx \frac{1}{n}, \quad n\in\NN.
\]
Hence, \ref{LorentzLRP:2} holds.
Pick $0<a<1/q$ such that \eqref{eq:LRP} holds. On one hand, if $t\le 2r+1$,
\[
A(r,t) \le \frac{1}{s_{r+1}} \left\Vert \sum_{n=r+1}^t \ee_n\right\Vert_{d_{1,q}(\ww)}\lesssim \frac{s_{t-r}}{s_{r+1}}\le 1.
\]
On the other hand, if $t\ge 2r$ using again \ref{LorentzLRP:1} we obtain
\[
A(r,t)
\approx \left(\sum_{k=r+1}^t \frac{s_{k-r}^q}{s_k^q(k-r)} \right)^{1/q}
\lesssim \left(\sum_{k=r+1}^t \frac{(k-r)^{aq}}{k^{aq}(k-r)} \right)^{1/q}.
\]
Hence, applying Lemma~\ref{lem:JarDif} yields the desired inequality.
\end{proof}

To contextualize the assumptions in Theorem~\ref{theoremLp} below we must take into account that any basis $\XB$ of a $r$-Banach space $\XX$, $0<r\le 1$, is dominated by the unit vector basis of the Lorentz sequence space $d_{1,r}(\ww)$, where the primitive weight of $\ww$ is $\udf[\XB,\XX]$ (see \cite{AABW2021}*{Theorem 9.12}). Although it is not central in our study, in the proof of Theorem~\ref{theoremLp} we will keep track of the \emph{quasi-greedy parameters} of the basis,
\[
\gc_m[\XB,\XX]=\sup\{ \Vert S_A[\XB,\XX](f) \Vert \colon A
\;\text{greedy set of}\;
f\in B_\XX, \, |A|= m\},
\]
where for a finite subset $A\subseteq \NN$, we let $S_A=S_A[\XB,\XX]\colon \XX \to\XX$ denote the coordinate projection on $A$ , i.e.,
\[
S_A(f)=\sum_{n\in A} \xx_n^{\ast}(f)\, \xx_n,\quad f\in \XX.
\]
The quasi-greedy parameters are bounded above by the \emph{unconditionality parameters}
\[
\unc_m=\unc_m[\XB,\XX] :=\sup_{|A|= m} \Vert S_A\Vert, \quad m\in\NN,
\]
which are used to quantify how far the basis is from being unconditional. Thus, the following result exhibits that bidemocratic bases are close to being quasi-greedy.

\begin{theorem}\label{thm:BidCond}
Let $\XB$ be a bidemocratic basis of a $p$-Banach space $\XX$, $0<p\le 1$. Then,
\[
\unc_m[\XB,\XX]\lesssim (\log m)^{1/p}, \quad m\ge 2.
\]
\end{theorem}

\begin{proof}
Just combine \cite{AABW2021}*{Proposition 5.7} with \cite{AAW2021b}*{Theorem 5.1}.
\end{proof}

Since $(\gc_m)_{m=1}^\infty$ need not be non-decreasing (see \cite{Oikhberg2018}*{Proposition 3.1}), we also set
\[
\g_m=\g_m[\XB,\XX]=\sup_{k\le m} \gc_k.
\]

Of course, $\XB$ is quasi-greedy if and only if $\sup_m \g_m=\sup_m \gc_m<\infty$, and $\XB$ is unconditional if and only if $\sup_m \unc_m<\infty$.

We will use the fact that quasi-greedy bases are in particular total bases (see \cite{AABW2021}*{Corollary 4.5}) to prove the advertised existence of bidemocratic non-quasi-greedy bases.

\begin{theorem}\label{theoremLp}
Let $1<q<\infty$, and let $\ww=(w_n)_{n=1}^\infty$ be a weight whose primitive weight $(s_m)_{m=1}^\infty$ is unbounded. Let $\XX$ be a quasi-Banach space with a basis $\XB$. Suppose that $\XB$ is bidemocratic with $\udf[\XB,\XX](m)\approx s_m$ for $m\in\NN$, and that $\XB$ has a subsequence dominated by the unit vector basis of $d_{1,q}(\ww)$. Then $\XX$ has a non-total bidemocratic basis $\YB$ with
\[
\udf[\YB,\XX](m)\approx s_m, \quad m\in \NN.
\]
Moreover, if $(s_m)_{m=1}^\infty$ has the LRP and $(s_m/m)_{m=1}^\infty$ is non-increasing,
\[
\gc_m[\YB,\XX] \gtrsim \left(\log m\right)^{{1/q'}}, \quad m\ge 2,
\]
where $1/q+1/q^{\prime}=1$.
\end{theorem}

\begin{proof}
Choose a subsequence $\left(\xx_{\eta(k)}\right)_{k=1}^{\infty}$ of $\XB=\left(\xx_n\right)_{n=1}^{\infty}$ so that $\eta(1)\ge 2$ and the linear operator $T\colon d_{1,q}(\ww)\rightarrow \XX$ given by
\[
T\left(\ee_k\right)= \xx_{\eta(k)}, \quad k\in\NN,
\]
is bounded. For each $n\in\NN$, $n\ge 2$, define $\yy_n=\xx_n+\zz_n$, where
\[
\zz_n
=\begin{cases}
w_k \, \xx_1& \text{ if } n=\eta(k),\\
0 & \text{otherwise}.
\end{cases}
\]
It is clear that $(\yy_n,\xx_n^{\ast})_{n=2}^\infty$ is a biorthogonal system. Thus, in order to prove that $\YB:=\left(\yy_n\right)_{n=2}^{\infty}$ is a basis of $\XX$ with dual basis $\YB^*=(\xx_n^*)_{n=2}^\infty$ it suffices to prove that
$\xx_1$ belongs to the closed linear span of $\YB$.
For each $m\in\NN$ we have
\begin{align*}
f_m&:=\frac{1}{H_m[\ww]} \sum_{k=1}^m \frac{1}{s_k} \yy_{\eta(k)}\\
&=\xx_1+\frac{1}{H_m[\ww]}\sum_{k=1}^m\frac{1}{s_k} \xx_{\eta(k)}\\
&=\xx_1+\frac{1}{H_m[\ww]} T(g_m),
\end{align*}
where $g_m=\sum_{k=1}^m s_k^{-1} \ee_k$. By \eqref{eq:NormLorentz},
\[
\Vert f_m-\xx_1\Vert\le \Vert T \Vert(H_m[\ww])^{-1/q'}, \quad m\in\NN.
\]
Since by Lemma~\ref{lem:AlsoDiverges}, $\lim_m H_m[\ww]=\infty$ we obtain $\lim_m f_m=\xx_1$. Since $\yy_n^{\ast}\left(\xx_1\right)=0$ for all $n\ge 2$, $\YB$ is not a total basis.

Put $\ZB=(\zz_n)_{n=2}^\infty$. Then,
\[
\udf[\ZB,\XX](m)=\Vert \xx_1\Vert \sum_{k=1}^m w_k \approx s_m, \quad m\in\NN.
\]
Hence,
\[
\udf[\YB,\XX](m)\lesssim \udf[\XB,\XX](m)+ \udf[\ZB,\XX](m)\lesssim s_m, \quad m\in\NN,
\]
so that, since $\udf[\YB^*,\XX](m)\lesssim m/s_m$ for $m\in\NN$, $\YB$ is bidemocratic.

In the case when $(s_m)_{m=1}^\infty$ has the LRP and $(s_m/m)_{m=1}^\infty$ is non-increasing, by parts \ref{LorentzLRP:1} and ~\ref{LorentzLRP:2} of Lemma~\ref{lem:LorentzLRP} we can assume without loss of generality that
\begin{equation}\label{eq:HarmonicEquivalence} H_t[\ww]-H_r[\ww]\approx H_t-H_r, \quad 0\le r \le t.
\end{equation}

To estimate the quasi-greedy parameters we appeal to Lemma~\ref{lem:Jar} to pick for each $m\ge 2$, natural numbers $r=r(m)$ and $s=s(m)$ with $m\le r \le s$, and
\begin{equation}\label{eq:HarmonicEstimates}
H_m[\ww] \le H_s[\ww]-H_r[\ww]\le (H_m[\ww])^{1/q}+H_m[\ww].
\end{equation}
Set $h_m=\sum_{k=r+1}^s s_k^{-1} \ee_k$ and
\begin{align*}
u_m&=\frac{1}{H_m[\ww]} \left( \sum_{k=1}^{m} \frac{1}{s_k} \yy_{\eta(k)} -\sum_{k=r+1}^{s} \frac{1}{s_k} \yy_{\eta(k)}\right)\\
&=\frac{1}{H_m[\ww]} \left( T(g_m)-T(h_m) + (H_m[\ww]-H_s[\ww]+H_r[\ww]) \xx_1 \right).
\end{align*}
By Lemma~\ref{lem:LorentzLRP}~\ref{LorentzLRP:3}, \eqref{eq:NormLorentz}, \eqref{eq:HarmonicEquivalence} and \eqref{eq:HarmonicEstimates},
\[
\max\{\Vert g_m\Vert, \Vert h_m\Vert, |H_s[\ww]-H_r[\ww]-H_m[\ww]|\}\lesssim H_m^{1/q}, \quad m\in\NN.
\]
Hence,
$
\Vert u_m\Vert \lesssim H_m^{-1/q'}
$
for $m\in\NN$. Since $A_m:=\{\eta(1),\dots,\eta(m)\}$ is a greedy set of $u_m$ with respect to $\YB$, and
\[
\Vert S_{A_m}[\YB,\XX](u_m)\Vert=\Vert f_m\Vert \approx 1,\quad m\in\NN,
\]
we are done.
\end{proof}

\begin{corollary}\label{corbidem}
Let $\XX$ be a Banach space with a Schauder basis. Suppose that $\XX$ has a complemented subspace isomorphic to $\ell_{p,q}$, where $p$, $q\in(1,\infty)$. Then $\XX$ has a non-total bidemocratic basis $\YB$ with
\[
\udf[\YB,\XX](m)\approx m^{{1/p}}, \quad m\in \NN,
\]
and
\[
\gc_m[\YB,\XX] \gtrsim \left(\log m\right)^{{1/q'}}, \quad m\ge 2.
\]
\end{corollary}

\begin{proof}
An application of the Dilworth-Kalton-Kutzarova method, or DKK-method for short (see \cites{AADK2019b,DKK2003}), yields a bidemocratic Schauder basis of $\XX$ with fundamental function equivalent to $(m^{1/p})_{m=1}^\infty$ (see \cite{AADK2019b}). The direct sum of this basis with the unit vector system of $\ell_{p,q}$ is a bidemocratic Schauder basis of $ \XX\oplus\ell_{p,q}\approx\XX$ that posesses a subsequence equivalent to the unit vector basis of $\ell_{p,q}$. Applying Theorem~\ref{theoremLp} we are done.
\end{proof}

Note that Corollary~\ref{corbidem} can be applied with $1<p=q<\infty$, so that $\ell_{p,q}=\ell_p$. Hence as a consequence we obtain the result that we announced in the Introduction.

\begin{theorem}\label{thm:BDNotTotal}
Let $1<p<\infty$. Then $\ell_p$ has a bidemocratic non-total (hence, non-quasi-greedy) basis.
\end{theorem}

Theorem~\ref{thm:BDNotTotal} leads us naturally to the question about the existence of bidemocratic non-total bases in $\ell_1$ and $c_0$. We make a detour from our route to solve both questions in the negative. For that we will need to apply the arguments that follow, keeping in mind that $\ell_1=(c_0)^{\ast}$ is a GT-space (see \cite{LinPel1968}).

\begin{proposition}
Let $\XX$ be a quasi-Banach space, and let $\XB=(\xx_n)_{n=1}^\infty$ and $\YB=(\xx_n^{\ast})_{n=1}^\infty$ be sequences in $\XX$ and $\XX^{\ast}$, respectively. Suppose that $(\xx_n,\xx_n^{\ast})_{n=1}^\infty$ is a biorthogonal system and that
\[
\udf[\XB,\XX](m) \, \udf[\YB,\XX^{\ast}](m)\le C m, \quad m\in\NN,
\]
for some constant $C$. Then $\Vert \Ind_{\varepsilon,A} [\XB,\XX] \Vert\le C \Vert f \Vert$ whenever $A\subseteq \NN$, $\varepsilon\in\EE_A$, and $f\in\XX$ are such that $|\xx_n^{\ast}(f)|\ge 1$ on a set of cardinality at least $|A|$.
\end{proposition}

\begin{proof}
In the case when $\XB$ spans the whole space $\XX$, this proposition says that any bidemocratic basis is truncation quasi-greedy. In fact, the proof of \cite{AABW2021}*{Proposition 5.7} gives this slightly more general result.
\end{proof}

\begin{theorem}\label{thm:AAW}
Let $\XX$ be a GT-space and let $\XB=(\xx_n)_{n=1}^\infty$ and $(\xx_n^{\ast})_{n=1}^\infty$ be a sequences in $\XX$ and $\XX^{\ast}$, respectively. Suppose that $(\xx_n,\xx_n^{\ast})_{n=1}^\infty$ is a biorthogonal system and that there is a constant $C$ such that $\Vert \Ind_{\varepsilon,A} [\XB,\XX] \Vert\le C \Vert f \Vert$ whenever $A\subseteq \NN$ and $f\in\XX$ satisfy $|\xx_n^{\ast}(f)|\ge 1 \ge |\xx_k^{\ast}(f)|$ for $(n,k)\in A\times (\NN\setminus A)$, and $\varepsilon=(\varepsilon_n)_{n\in A}\in\EE_A$ is defined by $\xx_n^{\ast}(f)=|\xx_n^{\ast}(f)|\, \varepsilon_n$. Then, $\ldf(m) \gtrsim m$ for $m\in\NN$.
\end{theorem}

\begin{proof}
In the case when $\XB$ spans the whole space $\XX$, this theorem says that any truncation quasi-greedy basis of a GT-space is democratic with fundamental function equivalent to $(m)_{m=1}^\infty$ (see \cite{AAW2021}*{Theorem 4.3}). As a matter of fact, the proof of \cite{AAW2021}*{Theorem 4.3} gives this slightly more general result.
\end{proof}

\begin{theorem}Let $\XB$ be a bidemocratic basis of a Banach space $\XX$.
\begin{enumerate}[label=(\roman*),leftmargin=*,widest=ii]
\item\label{GTspace}If $\XX$ is a GT-space, then $\XB$ is equivalent to the canonical basis of $\ell_1$.
\item\label{predualGTspace}If $\XX^{\ast}$ is a GT-space, then $\XB$ is equivalent to the canonical basis of $c_0$.
\end{enumerate}
\end{theorem}

\begin{proof}
Suppose that $\XX$ (resp.,\ $\XX^{\ast}$) is a GT-space. By Theorem~\ref{thm:AAW} $\ldf[\XB,\XX]$ (resp.,\ $\ldf[\XB^{\ast},\XX^{\ast}]$) is equivalent to $(m)_{m=1}^\infty$. Hence, $\udf[\XB^{\ast},\XX^{\ast}]$ (resp.,\ $\udf[\XB,\XX]$) is bounded. This readily gives that $\XB^{\ast}$ (resp.\ $\XB$) is equivalent to the canonical basis of $c_0$. To conclude the proof of \ref{GTspace}, we infer that $\XB^{**}$ is equivalent to the canonical basis $\BB_{\ell_1}$ of $\ell_1$. Since $\BB_{\ell_1}$ dominates $\XB$ and $\XB$ dominates $\XB^{**}$ we are done.
\end{proof}

It is known that some results involving the TGA work for total bases but break down if we drop this assumption (see, e.g., \cite{BL2020}*{Theorem 4.2 and Example 4.5}). In view of this, another question springing from Theorem~\ref{thm:BDNotTotal} is whether working with total bases makes a difference, i.e., whether bidemocratic total bases are quasi-greedy. We solve this question in the negative by proving the following theorem.

\begin{theorem}\label{thm:BDTotalNotQG}
Let $1<p<\infty$. Then any infinite-dimensional subspace of $\ell_p$ has a further subspace with a bidemocratic non-quasi-greedy total basis.
\end{theorem}

Theorem~\ref{thm:BDTotalNotQG} will follow as a consequence of the following general result.

\begin{theorem}\label{theoremLp2}
Let $\ww=(w_n)_{n=1}^\infty$ be a weight, and suppose that its primitive weight $(s_m)_{m=1}^\infty$ has the LRP and that $(s_m/m)_{m=1}^\infty$ is non-increasing. Let $\XX$ be a Banach space with a total basis $\XB$. Suppose that $\XB$ is bidemocratic with $\udf[\XB,\XX](m)\approx s_m$ for $m\in\NN$, and that $\XB$ has a subsequence dominated by the unit vector basis of $d_{1,q}(\ww)$ for some $q>1$. Then $\XX$ has a subspace $\YY$ with a basis $\YB$ satisfying the following properties:
\begin{enumerate}[label=(\roman*),leftmargin=*,widest=iii]
\item\label{bidemocraticlp2}$\YB$ is bidemocratic with $\udf[\YB,\YY](m)\approx s_m$ for $m\in\NN$.
\item\label{markulp2} $\YB$ is total.
\item\label{notqglps}$\YB$ is not quasi-greedy.
\item\label{notslp2}$\YB$ is not Schauder in any order.
\end{enumerate}
\end{theorem}

\begin{proof}
Choose a subsequence $\left(\xx_{\eta(j)}\right)_{j=1}^{\infty}$ of $\XB=\left(\xx_n\right)_{n=1}^{\infty}$ so that $\NN\setminus\eta(\NN)$ is infinite and the linear operator $T\colon d_{1,q}(\ww)\rightarrow \XX$ given by
\[
T\left(\ee_j\right)= \xx_{\eta(j)}, \quad k\in\NN,
\]
is bounded. Let $\psi\colon\NN\rightarrow\NN$ be the increasing sequence defined by $\psi(\NN)=\NN\setminus\eta(\NN)$. Since the harmonic series diverges we can recursively construct an increasing sequence $(t_k)_{k=0}^\infty$ of natural numbers with $t_0=0$ such that, if we put
\[
\Lambda_k=H_{t_k} -H_{t_{k-1}},
\]
then $\lim_k \Lambda_k=\infty$. For each $j\in\NN$ define $\yy_j=\xx_{\eta(j)}+\zz_j$, where
\[
\zz_j = \frac{s_j}{j} \xx_{\psi(k)}, \quad k\in\NN,\; t_{k-1} <j \le t_k.
\]
It is clear that $(\yy_j,\xx_{\eta(j)}^{\ast})_{j=1}^\infty$ is a biorthogonal system. Thus, to see that $\YB:=\left(\yy_j\right)_{j=1}^{\infty}$ satisfies \ref{bidemocraticlp2} it suffices to prove that, if $\ZB=(\zz_j)_{j=1}^\infty$, $\udf[\ZB,\XX](m) \lesssim s_m$ for $m\in\NN$. Set $C_1=\sup_n \Vert \xx_n\Vert$. For every $A\subseteq\NN$ with $|A|=m<\infty$ and $\varepsilon\in\EE_A$ we have
\[
\Vert \Ind_{\varepsilon,A}[\ZB,\XX]\Vert \le C_1 \sum_{j\in A} \frac{s_j}{j} \le C_1\sum_{j=1}^m \frac{s_j}{j} \lesssim s_m.
\]

Let us see that $\YB$ is a total basis of $\YY=[\YB]$. Set
\[
\zz_k^{\ast}=\xx_{\psi(k)}^{\ast} -\sum_{j=1+t_{k-1}}^{t_k} \frac{s_j}{j} \, \xx^{\ast}_{\eta(j)}, \quad k\in\NN.
\]
We have $\zz_k^{\ast}(\yy_j)=0$ for all $j$ and $k\in\NN$. Therefore $\zz_k^{\ast}(f)=0$ for all $f\in\YY$ and $k\in\NN$. Pick $f\in\YY$ and suppose that $\xx_{\eta(j)}^{\ast}(f)=0$ for all $j\in\NN$. We infer that $\xx_{\psi(k)}^{\ast}(f)=0$ for all $k\in\NN$. Since $\XB$ is a total basis, $f=0$.

To prove that $\YB$ is neither a quasi-greedy basis nor a Schauder basis in any ordering, we pick a permutation $\pi$ of $\NN$. For each $k\in\NN$, choose $A_k\subseteq D_k:=[1+t_{k-1},t_k]\cap \NN$ minimal with the properties
\[
l:=\max( \pi^{-1}(A_k)) <\min (\pi^{-1}(D_k\setminus A_k))
\;\text{and}\;
\Gamma_k:=\sum_{j\in A_k} \frac{1}{j} > \frac{\Lambda_k}{2}.
\]
By construction,
\[
\frac{\Lambda_k}{2} \ge \Gamma_k-\frac{1}{\pi(l)} \ge \Gamma_k-1.
\]
Then, if we set
\[
\Theta_k:=\sum_{j\in D_k\setminus A_k} \frac{1}{j}=\Lambda_k-\Gamma_k,
\]
we have
$\Gamma_k-\Theta_k=-\Lambda_k+2\Gamma_k\in(0,2]$. Also by construction, if we set
\[
g_k=\sum_{j\in A_k} \frac{1}{s_j} \yy_j, \quad h_k=\sum_{j\in D_k\setminus A_k} \frac{1}{s_j} \yy_j, \quad k\in\NN,
\]
then $g_k$ is a partial-sum projection of $f_k:=g_k-h_k$ with respect to the rearranged basis $(\yy_{\pi(i)})_{i=1}^\infty$. Moreover, in the case when $\pi$ is the identity map, $g_k$ is a greedy projection of $f_k$. On one hand, if we set
\[
f_k'= \sum_{j\in A_k} \frac{1}{s_j} \ee_j- \sum_{j\in D_k\setminus A_k} \frac{1}{s_j} \ee_j,\\
\]
we have $f_k=T(f_k') + (\Gamma_k-\Theta_k) \xx_{\psi(k)}$ for all $k\in\NN$. By Lemma~\ref{lem:LorentzLRP}~\ref{LorentzLRP:3},
\[
\Vert f_k'\Vert_{d_{1,q}(\ww)} = \left\Vert \sum_{j=1+t_{k-1}}^{t_k} \frac{1}{s_j} \ee_j\right\Vert_{d_{1,q}(\ww)}\lesssim
\max\{ 1, \Lambda_k^{1/q}\}\approx \Lambda_k^{1/q}.
\]
Hence, $\Vert f_k\Vert \lesssim \Lambda_k^{1/q}$ for $k\in\NN$. On the other hand, since $\xx_{\psi(k)}^{\ast} (g_k)=\Gamma_k$, we have
\[
\Lambda_k<2\Gamma_k\le 2 C_2 \Vert g_k\Vert
\]
where $C_2=\sup_n \Vert \xx_n^{\ast}\Vert$. Summing up,
\[
\frac{ \Vert g_k\Vert}{\Vert f_k\Vert} \gtrsim \Lambda_k^{1/q'} \xrightarrow[k\to \infty]{}\infty. \qedhere
\]
\end{proof}

\begin{corollary}\label{corollaryl2}
There is a bidemocratic total basis of $\ell_2$ that is not Schauder under any rearrangement of the terms nor quasi-greedy.
\end{corollary}

Let us notice that the bases we construct to prove Theorem~\ref{thm:BDTotalNotQG} are \emph{not} Schauder bases. As the TGA does no depend on the particular way we reorder the basis, whereas being a Schauder basis does, studying the TGA within the framework of Schauder bases is somehow unnatural. Nonetheless, Schauder bases have provided a friendly framework to develop the greedy approximation theory with respect to bases since its beginning at the turn of the century. In fact, it is nowadays unknown even whether certain results involving the TGA work outside the framework of Schauder bases (see, e.g., \cite{Berna2020})! Hence, in connection with our discussion it is natural to wonder whether bidemocratic Schauder bases are quasi-greedy. We close this section by providing a negative answer to this question too.

\begin{theorem}\label{thm:BDSchauderNotQG}
There is a Banach space with a bidemocratic Schauder basis which is not quasi-greedy.
\end{theorem}

The proof of Theorem~\ref{thm:BDSchauderNotQG} relies on a construction that has its roots in \cite{KoTe1999}, where it was used to build a conditional quasi-greedy basis. Variants of the original idea of Konyagin and Telmyakov have appeared in several papers with different motivations (see \cites{GHO2013,BBGHO2018,AABW2021,Oikhberg2018}). Prior to tackling the proof we introduce a quantitative version of \cite{DKKT2003}*{Theorem 5.4}.

\begin{theorem}\label{thm:dualQuantitative}
Let $\XB$ be a bidemocratic basis of a quasi-Banach space $\XX$. Then
\[
\gc_m[\XB^{\ast},\XX^{\ast}]\lesssim \gc_m[\XB,\XX], \quad m\in\NN.
\]
In particular, if $\XB$ is a Schauder basis of a Banach space, then
\[
\gc_m[\XB^{\ast},\XX^{\ast}]\approx \gc_m[\XB,\XX], \quad m\in\NN.
\]
\end{theorem}

\begin{proof}
The proof of \cite{DKKT2003}*{Theorem 5.4} (see also \cite{AABW2021}*{Proof of Proposition 5.7}) yields the first estimate. To see the equivalence, we use that $\XB^{**}$ is equivalent to $\XB$ (see \cite{AlbiacKalton2016}*{Corollary 3.2.4}).
\end{proof}

\begin{proposition}\label{theorembidemocraticnotqG}
Let $1<p<\infty$. There is a Banach space $\XX$ with a monotone Schauder basis $\XB$ with the following properties:
\begin{enumerate}[label=(\roman*),leftmargin=*,widest=ii]
\item\label{1bidem}For all finite sets $A\subseteq \NN$ and all $\varepsilon\in \EE_A$,
\[
\Vert \Ind_{\varepsilon A}\Vert=\left|A\right|^{1/p}
\;\text{and}\;
\Vert\Ind_{\varepsilon A}^{\ast}\Vert=\left|A\right|^{1/p'},
\]
where $1/p+1/p^{\prime}=1$. Therefore, $\XB$ is $1$-bidemocratic.
\item\label{notnQG} Neither $\XB$ nor $\XB^{\ast}$ are quasi-greedy. Quantitatively,
\[
\gc_m\approx\gc_m^{\ast}\approx \unc_m\approx \unc_m^{\ast}\approx \left(\log{m}\right)^{1/p'} , \quad m\in\NN,\; m\ge 2.
\]
\end{enumerate}
\end{proposition}
\begin{proof}
Put
\[
\DD:=\{(m,k)\in\NN^2 \colon 1\le k \le m\},
\]
where the elements are taken in the lexicographical order. Appealing to Lemma~\ref{lem:Jar} we recursively construct a family $(r_{m,k},s_{m,k})_{(m,k)\in\DD}$ in $\NN^2$ such that
\begin{align}
m+1<r_{m,k}<s_{m,k},\quad&1\le k\le m,\label{movetotheright1}\\
s_{m,k}<r_{m,k+1}, \quad &1\le k<m,
\;\text{and}\label{movetotheright2}\\
\frac{1}{k}- \frac{1}{m}\le T_{m,k}:= \sum_{j=r_{m,k}}^{s_{m,k}}\frac{1}{j}<\frac{1}{k},\quad &1\le k\le m. \label{conditionclosesums}
\end{align}

Next, we choose a sequence $( A_m)_{m=1}^\infty$ of integer intervals contained in $\NN$ so that $\max(A_m)<\min(A_{m+1})$ for all $m\in\NN$, and
\begin{equation}\label{separatingthesets}
|A_m|=2m+\sum_{k=1}^{m}s_{m,k}-r_{m,k}.
\end{equation}
Let
\[
i_{m,k}=\min A_m+\sum_{j=1}^{k-1}\left(s_{m,j}-r_{m,j}+2\right), \quad (m,k)\in\DD.
\]
Fix $m\in\NN$. For each $n\in A_m$ there are unique integers $1\le k \le m$ and $-1\le t \le s_{m,k}-r_{m,k}$ so that
$n=i_{m,k}+1+t$. Let us set
\[
(d_{m,n},\varepsilon_{m,n})=
\begin{cases}
(k,1) &\text{ if } t=-1,\\
(r_{m,k}+t,-1) &\text{otherwise.}
\end{cases}
\]
Consider the subset of $\NN$ given by
\begin{equation*}
B_m=\{ i_{m,k} \colon 1\le k \le m\}.
\end{equation*}
The family $(d_{m,n})_{n\in B_m}$ is increasing. By \eqref{movetotheright2}, $(d_{m,n})_{n\in A_m\setminus B_m}$ is also increasing, and by \eqref{movetotheright1},
\begin{equation}\label{eq:BlGreedy}
\max_{n\in B_m} d_{m,n} < \min_{n\in A_m\setminus B_m} d_{m,n}.
\end{equation}
Set $b_{m,n}=d_{m,n}^{-1/p'}$ for $m\in\NN$ and $n\in A_m$. Since the family $(d_{m,n})_{n\in A_m}$ consists of distinct positive integers, for each $m\in \NN$ and $A\subseteq A_m$ we have
\begin{align}
\sum_{n\in A}b_{m,n}
&\le \sum_{n=1}^{\left|A\right|}n^{-1/p'}
\le p \left|A\right|^{1/p},
\;\text{and}\label{firstboundforfundamentalfunction}\\
\sum_{n\in A}b_{m,n}^{p'}
&\le H_{|A|},
\label{lastestimate}
\end{align}
where, as we said, $H_m$ denotes the $m$th harmonic number. Once the family $(b_{m,n})_{m\in\NN,n\in A_m}$ has been constructed, we define $\Vert\cdot\Vert_{\maltese}$ on $c_{00}$ by
\[
\left\Vert(a_n)_{n=1}^\infty\right\Vert_{\maltese}=\frac{1}{p}\sup_{\substack{m\in\NN\\ l\in A_m}}\left|\sum_{\substack{n\in A_m\\n\le l}}a_{n}b_{m,n} \right|.\label{Snorm0}
\]
Since $\max(A_m)<\min(A_{m+1})$ for all $m\in \NN$, we have that $\Vert f\Vert_{\maltese}<\infty$ for all $f\in c_{00}$, so that $\Vert\cdot\Vert_{\maltese}$ is a semi-norm. Let $\XX$ be the Banach space obtained as the completion of $c_{00}$ endowed with the norm
\[
\Vert f\Vert=\max\left\lbrace \Vert f\Vert_p,\Vert f\Vert_{\maltese}\right\rbrace.\label{norm}
\]
It is routine to check that the unit vector system $\XB$ is a monotone normalized Schauder basis of $\XX$ whose coordinate functionals $\XB^{\ast}$ are the canonical projections on each coordinate. It follows from \eqref{firstboundforfundamentalfunction} that
\[
\Vert \Ind_{\varepsilon,A}\Vert_{\maltese}\le \frac{1}{p}\sup_{m\in \NN}\sum_{n\in A\cap A_m} b_{m,n}\le \left|A\right|^{1/p},
\quad |A|<\infty, \; \varepsilon\in\EE_A.
\]
By definition, there is a norm-one linear map from $\XX$ into $\ell_p$ which maps $\XB$ to the unit vector system of $\ell_p$. By duality, there is a norm-one map from $\ell_{p'}$ into $\XX^{\ast}$ which maps the unit vector system of $\ell_{p'}$ to $\XB^{\ast}$.
In particular,
\[
\Vert \Ind_{\varepsilon,A}^{\ast}\Vert \le |A|^{1/p'}, \quad |A|<\infty, \; \varepsilon\in\EE_A.
\]
We infer that \ref{1bidem} holds.

Define $a_{m,n}=\varepsilon_{m,n} d_{m,n}^{-1/p}$, so that $a_{m,n} b_{m,n}={\varepsilon_{m,n}}/{d_{m,n}}$ for $m\in\NN$ and $n\in A_m$. For each $m\in\NN$ set
\[
f_m=\sum_{n\in A_m}a_{m,n}\, \xx_n.
\]
Let $(m,k)\in\DD$ and use the convention $i_{m,m+1}=1+\max(A_m)$. If $i_{m,k}\le l <i_{m,k+1}$ by
construction we have
\[
B_{m,k}(l):=\sum_{n=i_{m,k}}^{l} a_{m,n} b_{m,n}=\frac{1}{k}-\sum_{j=r_{m,k}}^{l-1+r_{m,k}-i_{m,k}} \frac{1}{j}.
\]
Thus, the maximum and minimum values of $B_{m,k}(l)$ on the interval $i_{m,k}\le l <i_{m,k+1}$
are $1/k$ and $1/k-T_{m,k}$, respectively. Since by the right hand-side inequality in \eqref{conditionclosesums}, $1/j-T_{m,j}>0$ for all $1\le j \le m$ we infer that
\[
\|f_m\|_{\maltese}
=\frac{1}{p} \max_{l\in A_m} \sum_{\substack{n\in A_m \\ n\le l}} a_{m,n} b_{m,n}
=\frac{1}{p}\max_{1\le k \le m} \frac{1}{k}+ \sum_{j=1}^{k-1} \left( \frac{1}{j}-T_{m,j}\right).
\]
Using the left hand-side inequality in \eqref{conditionclosesums} we obtain
\[
\|f_m\|_{\maltese}\le \frac{1}{p} \max_{1\le k \le m} \frac{1}{k}+\frac{k-1}{m}=\frac{1}{p}.
\]

We also have
\[
\Vert f_m\Vert_{p}^p=\sum_{k=1}^m\left( \frac{1}{k}+T_{m,k}\right) \le 2 H_m.
\]
Hence, $\Vert f_m\Vert\le 2^{1/p} H_m^{1/p}$ for all $m\in\NN$.

By \eqref{eq:BlGreedy}, $B_m$ is a greedy set of $f_m$. Since every coefficient of $f_m$ is positive on $B_m$,
\[
\Vert S_{B_m}(f_m) \Vert\ge \Vert S_{B_m}(f_m) \Vert_{\maltese}=\frac{1}{p}\sum_{j\in B_m} \frac{1}{d_{m,n}}=\frac{1}{p}H_m.
\]
Summing up,
\[
\frac{\Vert S_{B_m}(f_m)\Vert} {\Vert f_m\Vert}\ge \frac{1}{p \, 2^{1/p}} H_m^{1/p'}, \quad m\in\NN.
\]
Since $|B_m|=m$, this shows that $\g_m\ge p^{-1} 2^{-1/p} H_m^{1/p'}$ for all $m\in\NN$.

By Theorem~\ref{thm:dualQuantitative}, it only remains to obtain the upper estimate for the unconditionality constants of $\XB$. By \eqref{lastestimate} and H\"older's inequality, for all $A\subseteq\NN$ with $|A|\le m$ we have
\[
\Vert S_A( f)\Vert_{\maltese}\le \frac{1}{p}\Vert f\Vert_p \sup_m\left( \sum_{n\in A\cap A_m} |b_{m,n}|^{p'}\right)^{1/p'}
\le \frac{1}{p} H_m^{1/p'} \Vert f\Vert_p.
\]
Hence, $\unc_m\le\max\{1, H_m^{1/p'}/p \}$ for all $m\in\NN$.
\end{proof}

\begin{remark}
Given a basis $\XB$ and an infinite subset $\n$ of $\NN$, we say that $\XB$ is $\n$-quasi-greedy if
\[
\sup\left\lbrace \dfrac{\Vert S_A(f)\Vert}{\Vert f\Vert} \colon f\in\XX,\, A
\;\text{greedy set of}\;
f,\, |A|\in\n\right\rbrace<\infty
\]
(see \cite{Oikhberg2018}). Note that the basis constructed in Proposition~\ref{theorembidemocraticnotqG} is not $\n$-quasi-greedy for any increasing sequence $\n$.
\end{remark}

\begin{remark}\label{remarklp}
The basis $\XB$ in Proposition~\ref{theorembidemocraticnotqG} has a subbasis isometrically equivalent to the unit vector basis of $\ell_p$. Indeed, it is easy to check that $\left(\xx_{i_{m,1}}\right)_{m=1}^\infty$ has this property. The basis $\XB$ also has, as we next show, a block basis isometrically equivalent to the unit vector basis of $c_0$. Let $(A_m)_{m=1}^\infty$, $(B_m)_{m=1}^\infty$ and $(f_m)_{m=1}^\infty$ be as in that proposition, and define
\[
g_m:=S_{B_m}(f_m), \quad h_m=\frac{g_m}{\|g_m\|_{\maltese}}, \quad m\in\NN.
\]
Pick positive scalars $(\varepsilon_k)_{k=1}^\infty$ with $\sum_{k=1}^\infty \varepsilon_k^p=1$.
Since
\[
\lim_{m} \frac{ \Vert g_m\Vert_p}{\Vert g_m\Vert_{\maltese}}=0,
\]
there is a subsequence $\left(g_{m_k}\right)_{k=1}^\infty$ with $\left\Vert g_{m_k}\right\Vert_p\le \varepsilon_k \Vert g_{m_k}\Vert_{\maltese}$ for all $k\in\NN$. Let $f=(a_k)_{k=1}^\infty\in c_{00}$. Since $\supp(h_m)\subseteq A_{m}$ for all $m$, we have
\[
\left\Vert \sum_{k=1}^{\infty}a_k h_{m_k} \right\Vert_{\maltese}
=\max_{k\in\NN} |a_k| \Vert h_{m_k}\Vert_{\maltese}
=\max_{k\in\NN}|a_k|
\]
and
\[
\left\Vert \sum_{k=1}^{\infty}a_k h_{m_k} \right\Vert_p
=\left( \sum_{k=1}^\infty |a_k|^p \Vert h_{m_k}\Vert_p^p\right)^{1/p}
\le \left( \sum_{k=1}^\infty |a_k|^p \varepsilon_k^p \right)^{1/p}
\le \max_{k\in\NN} |a_k|.
\]
Consequently, $\Vert \sum_{k=1}^\infty a_k h_{m_k}\Vert= \max_{k\in\NN} |a_k|$.
\end{remark}

\section{Building bidemocratic conditional quasi-greedy bases}\label{sect:NM}\noindent
Probably, the most versatile method for building conditional quasi-greedy bases is the previously mentioned DKK-method due to Dilworth, Kalton and Kutzarova, which works only in the locally convex setting (i.e., for Banach spaces). It produces conditional almost greedy bases whose fundamental function either is equivalent to $(m)_{m=1}^\infty$ or has both the LRP and the URP. Thus, the DKK-method serves as a tool for constructing Banach spaces with bidemocratic conditional quasi-greedy bases whose fundamental function has both the LRP and the URP. In this section we develop a new method for building conditional bases that allows us to construct bidemocratic conditional quasi-greedy bases with an arbitrary fundamental function.

We write $\XX\oplus\YY$ for the Cartesian product of the quasi-Banach spaces $\XX$ and $\YY$ endowed with the quasi-norm
\[
\Vert (f,g)\Vert=\max\{ \Vert f\Vert, \Vert g\Vert\}, \quad f\in\XX,\, g\in\YY.
\]
Given sequences $\XB=(\xx_n)_{n=1}^\infty$ and $\YB=(\yy_n)_{n=1}^\infty$ in quasi-Banach spaces $\XX$ and $\YY$ respectively, its direct sum is the sequence $\XB\oplus\YB=(\uu_n)_{n=1}^\infty$ in $\XX\oplus\YY$ given by
\[
\uu_{2n-1}=(\xx_n,0), \quad \uu_{2n}=(0,\yy_n), \quad n\in\NN.
\]
If $\XB$ and $\YB$ are bidemocratic bases, and $\udf[\XB,\XX]\approx \udf[\YB,\YY]$, then the basis $\XB\oplus \YB$ of $\XX\oplus \YY$ is also bidemocratic with
\begin{align*}
\udf[\XB\oplus \YB,\XX\oplus \YY]&\approx \udf[\XB,\XX]\approx \udf[\YB,\YY],\\
\g_m[\XB\oplus \YB,\XX\oplus \YY]=&\max\{\g_m[\XB,\XX], \g_m[\YB,\YY]\},\\
\unc_m[\XB\oplus \YB,\XX\oplus \YY]=&\max\{\unc_m[\XB,\XX], \unc_m[\YB,\YY]\}.
\end{align*}

Loosely speaking, we could say that $\XB\oplus \YB$ inherits naturally the properties of $\XB$ and $\YB$. In contrast, `rotating' $\XB\oplus \YB$ gives rise to more interesting situations. In this section we study the `rotated' sequence $\XB\diamond\YB=(\zz_n)_{n=1}^\infty$ in $\XX\oplus\YY$ given by
\[
\zz_{2n-1}=\frac{1}{\sqrt{2}}(\xx_n,\yy_n), \quad \zz_{2n}=\frac{1}{\sqrt{2}}(\xx_n,-\yy_n), \quad n\in\NN.
\]
Note that
\[
\sum_{n=1}^\infty a_n\, \zz_n =\frac{1}{\sqrt{2}}\left(\sum_{n=1}^\infty (a_{2n-1}+a_{2n})\xx_n, \sum_{n=1}^\infty (a_{2n-1}-a_{2n})\yy_n\right),
\]
whenever the series converges.

To deal with bases built using this method, we introduce some notation. Given $A\subseteq\NN$ we set
\[
A^{\plus}=\{2n-1\colon n\in A\}, \quad A^{\minus}=\{2n\colon n\in A\}.
\]
Consider also the onto map $\eta\colon\NN\to\NN$ given by $\eta(n)=\lceil n/2 \rceil$. Note that $\eta^{-1}(A)=A^{\plus}\cup A^{\minus}$ and $\eta(A^{\plus})=\eta( A^{\minus})=A$ for all $A\subseteq\NN$.

Our first auxiliary result is pretty clear and well-known. In its statement we implicitly use the natural identification of $(\XX\oplus\YY)^{\ast}$ with $\XX^{\ast}\oplus\YY^{\ast}$.

\begin{lemma}[cf.\ \cite{AAW2019}*{Theorem 2.6}]\label{lem:dualdiamond}
Suppose that $\XB$ and $\YB$ are bases of $\XX$ and $\YY$ respectively. Then $\XB\diamond\YB$ is a basis of $\XX\oplus\YY$ whose dual basis is $\XB^{\ast}\diamond\YB^{\ast}$. Moreover, if $\XB$ and $\YB$ are Schauder bases, so is $\XB\diamond\YB$.
\end{lemma}

\begin{lemma}\label{constantupperdemfunct}
Let $\XB$ be a basis of a quasi-Banach space. There is a constant $C$ such that
\begin{equation*}
\left\Vert \sum_{n=1}^\infty a_n\, \xx_n \right\Vert \le C \udf(m)
\end{equation*}
whenever $|a_n|\le 1$ for all $n\in\NN$ and $a_n\not=0$ for at most $m$ indices. Moreover, if $\XX$ is $p$-Banach space, $0<p\le 1$, we can choose $C=(2^p-1)^{-1/p}$.
\end{lemma}

\begin{proof}
It follows readily from \cite{AABW2021}*{Corollary 2.3}.
\end{proof}

\begin{lemma}\label{lem:SDDiamond}
Suppose that $\XB$ and $\YB$ are bases of $\XX$ and $\YY$ respectively. Then
\[
\udf[\XB\diamond\YB,\XX\oplus\YY]
\le
C\max\{ \udf[\XB,\XX], \udf[\YB,\YY]\}
\]
for some constant $C$ that only depends of the spaces $\XX$ and $\YY$ (and it is $\sqrt{2}$ if $\XX$ and $\YY$ are Banach spaces).
\end{lemma}

\begin{proof}
Let $m\in\NN$, $A\subseteq\NN$ with $|A|\le m$, and $\varepsilon=(\varepsilon_n)_{n\in A}\in\EE_A$. We extend $\varepsilon$ by setting $\varepsilon_n=0$ if $n\in\NN\setminus A$. Put
\[
B=\{ n\in\NN \colon 2n-1\in A\} \cup \{ n\in\NN \colon 2n\in A\},
\]
that is, $B=\eta(A)$. We have $|\varepsilon_{2n-1} \pm \varepsilon_{2n}|\le 2 \chi_B(n)$ for all $n\in\NN$, and all $B$ such that $|B|\le |A|$. Thus, if $C$ is the constant in Lemma~\ref{constantupperdemfunct},
\begin{align*}
\Vert \Ind_{\varepsilon,A}&[\XB\diamond\YB,\XX\oplus\YY]\Vert\\
&= \frac{1}{\sqrt{2}} \max \left\{
\left\Vert \sum_{n=1}^\infty (\varepsilon_{2n-1}+\varepsilon_{2n})\xx_n \right\Vert,
\left\Vert \sum_{n=1}^\infty (\varepsilon_{2n-1}-\varepsilon_{2n})\yy_n \right\Vert
\right\}\\
&\le \frac{2C}{\sqrt{2}} \max \left\{ \udf[\XB,\XX](m), \udf[\YB,\YY](m).
\right\}\qedhere
\end{align*}
\end{proof}

\begin{proposition}\label{prop:bidem}
Suppose that $\XB$ and $\YB$ are bidemocratic bases of quasi-Banach spaces $\XX$ and $\YY$ respectively. Suppose also that
\[
s_m:= \udf[\XB,\XX](m) \approx \udf[\YB,\YY](m), \quad m\in\NN.
\]
Then $\XB\diamond\YB$ is a bidemocratic basis of $\XX\oplus\YY$. Moreover,
\[
\udf[\XB\diamond\YB,\XX\oplus\YY](m) \approx s_m, \quad m\in\NN.
\]
\end{proposition}

\begin{proof}Since, by assumption,
\[
\max\{\udf[\XB^{\ast},\XX](m) , \udf[\YB^{\ast},\YY](m)\} \lesssim \frac{m}{s_m}, \quad m\in\NN,
\]
applying Lemma~\ref{lem:SDDiamond} yields
\[
\udf[\XB\diamond\YB,\XX\oplus\YY](m) \lesssim s_m, \quad m\in\NN,\]
and
\[
\udf[\XB^{\ast}\diamond\YB^{\ast},\XX^{\ast}\oplus\YY^{\ast}](m) \lesssim \frac{m}{s_m}, \quad m\in\NN.
\]
Using Lemma~\ref{lem:dualdiamond}, these inequalities readily give the desired result.
\end{proof}

\begin{proposition}\label{prop:diamondConditional}
Let $\XB=(\xx_n)_{n=1}^\infty$ and $\YB=(\yy_n)_{n=1}^\infty$ be non-equivalent bases of quasi-Banach spaces $\XX$ and $\YY$ respectively. Then, $\XB\diamond\YB$ is a conditional basis of $\XX\oplus\YY$. Quantitatively, if
\[
\textstyle
\co_m=\{(a_n)_{n=1}^\infty\in\FF^\NN \colon |\{n\in\NN\colon a_n\not=0\}|\le m\}
\]
and
\[
E_m[\XB,\YB]=\sup_{(a_n)_{n=1}^\infty\in \co_m }
\frac{\Vert \sum_{n=1}^\infty a_n\, \xx_n\Vert}{\Vert \sum_{n=1}^\infty a_n\, \yy_n\Vert}, \quad m\in\NN,
\]
then
\[
\unc_m[\XB\diamond\YB,\XX\oplus\YY]\ge \frac{1}{2}\max\{E_m[\XB,\YB],E_m[\YB,\XB]\}, \quad m\in\NN.
\]
\end{proposition}

\begin{proof}
Given an eventually null sequence $(a_n)_{n\in A}$, define $(b_n)_{n\in A^{\plus}}$ and $(c_n)_{n\in A^{\minus}}$ by $b_{2n-1}=c_{2n}=a_n$ for all $n\in A$. If
\[
f_{\plus}=\sum_{n\in A^{\plus}} b_n\, \zz_n
\;\text{and}\;
f_{\minus}=\sum_{n\in A^{\minus}} c_n\, \zz_n
\]
we have
\begin{align*}
f_{\plus}
&=\frac{1}{\sqrt{2}}\left(\sum_{n\in A} a_n\, \xx_n,\sum_{n\in A} a_n\, \yy_n\right),\\
f_{\plus}+ f_{\minus} &=\sqrt{2} \left(\sum_{n\in A} a_n\, \xx_n,0\right),
\;\text{and}\\
f_{\plus}- f_{\minus}&=\sqrt{2} \left(0,\sum_{n\in A} a_n\, \yy_n\right).
\end{align*}
Therefore,
\[
\frac{ \Vert f_{\plus}\Vert}{ \Vert f_{\plus}+f_{\minus}\Vert }\ge\frac{1}{2}\frac{\|\sum_{n=1}^{\infty}a_n\yy_n\|}{\| \sum_{n=1}^\infty a_n\, \xx_n\|}
\;\text{and}\;
\frac{ \Vert f_{\plus}\Vert}{ \Vert f_{\plus}-f_{\minus}\Vert }\ge\frac{1}{2}\frac{\|\sum_{n=1}^{\infty}a_n\xx_n\|}{\| \sum_{n=1}^\infty a_n\, \yy_n\|}.\qedhere
\]
\end{proof}

Proposition~\ref{prop:diamondConditional} gives that the conditionality constants of $\XB\diamond\YB$ are bounded below by
\[
\frac{1}{2}\max\left\{
\frac{\udf[\XB,\XX]}{\udf[\YB,\YY]},
\frac{\udf[\YB,\YY]}{\udf[\XB,\XX]},
\frac{\ldf[\XB,\XX]}{\ldf[\YB,\YY]},
\frac{\ldf[\YB,\YY]}{\ldf[\XB,\XX]}
\right\}.
\]
Thus, applying our method to bases with non-equivalent fundamental functions yields `highly' conditional bases. In contrast, since bidemocratic bases are truncation quasi-greedy (see \cite{AABW2021}*{Proposition 5.7}), a combination of Proposition~\ref{prop:bidem} with Theorem~\ref{thm:BidCond} exhibits that we can apply the `rotation method' to bidemocratic bases with equivalent fundamental functions to obtain bases whose conditionality constants grow `slowly'. However, the basis $\XB\diamond\YB$ is always conditional unless $\XB$ and $\YB$ are equivalent. In this context, since quasi-greedy bases are truncation quasi-greedy (see \cite{AABW2021}*{Theorem 4.13}) we ask ourselves whether our construction preserves quasi-greediness. Our next result provides an affirmative answer to this question.

\begin{theorem}\label{thm:QGC}
Let $\XB$ and $\YB$ be bidemocratic bases of quasi-Banach spaces $\XX$ and $\YY$ respectively. Suppose that
\[
\udf[\XB,\XX](m) \approx \udf[\YB,\YY](m), \quad m\in\NN.
\]
Then,
\[
\g_m[\XB\diamond\YB,\XX\oplus\YY]\approx \max\{\g_m[\XB,\XX], \g_m[\YB,\YY]\}, \quad m\in\NN.
\]
In particular, $\XB\diamond\YB$ is quasi-greedy if and only if $\XB$ and $\YB$ are quasi-greedy.
\end{theorem}

Before the proof of Theorem~\ref{thm:QGC} we give two auxiliary lemmas.

\begin{lemma}\label{lem:ShareBidem}
Let $\XB=(\xx_n)_{n=1}^\infty$ and $\YB=(\yy_n)_{n=1}^\infty$ be bases of quasi-Banach spaces $\XX$ and $\YY$ respectively. Suppose that $\YB$ is truncation quasi-greedy and that
\[
\udf[\XB,\XX](m) \lesssim \ldf[\YB,\YY](m), \quad m\in\NN.
\]
Then, there is a constant $C_0$ such that
\[
\left\Vert \sum_{n\in E} c_n \, \xx_n \right\Vert
\le C_0 \left\Vert \sum_{n=1}^\infty d_n\, \yy_n \right\Vert
\]
whenever $\max_{n\in E} |c_n|\le \min_{n\in E} |d_n|$.
\end{lemma}

\begin{proof}
It is immediate from Lemma~\ref{constantupperdemfunct} and Lemma~\ref{lem:truncation quasi-greedyQU} combined.
\end{proof}

\begin{lemma}\label{cor:doubling}
Let $\XB$ be a basis of a quasi-Banach space $\XX$. If $\XB$ is truncation quasi-greedy there is a constant $C$ such that
\[
\gc_m[\XB,\XX]\le C\, \gc_k[\XB,\XX], \quad k\le m \le 2k.
\]
In particular, the sequences $(\gc_m[\XB,\XX])_{m=1}^\infty$ and $(\g_m[\XB,\XX])_{m=1}^\infty$ are doubling.
\end{lemma}

\begin{proof}
Let $A$ be a greedy set of $f\in\XX$ with $|A|=m$. Pick a greedy set $B$ of $f$ with $B\subseteq A$ and $|B|=k$. Since $|A\setminus B|\le |B|$, applying Lemma~\ref{lem:ShareBidem} with $\XB$ and a permutation of $\XB$ yields $\Vert S_{A\setminus B}(f) \Vert \le C_0 \Vert f\Vert$, where $C_0$ only depends on $\XB$ and $\XX$. Hence, if $\kappa$ denotes the modulus of concavity of $\XX$, $\Vert S_A(f)\Vert \le \kappa(C_0+\gc_k) \Vert f \Vert$.
\end{proof}

\begin{proof}[Proof of Theorem~\ref{thm:QGC}]
Set $\XB\diamond \YB=(\zz_n)_{n=1}^\infty$. Choosing $f\in\XX\oplus\YY$ with $\zz_{2n-1}^{\ast}(f)=\pm \zz_{2n}^{\ast}(f)$ yields
\[
\h_m:= \max\{\g_m[\XB,\XX], \g_m[\YB,\YY]\}\le \g_{2m}[\XB\diamond\YB,\XX\oplus\YY], \quad m\in\NN.
\]
Using Lemma~\ref{cor:doubling}, we obtain the desired upper estimate for $\h_m$.

For $A\subseteq\NN$, set $S_A=S_A[\XB\diamond\YB,\XX\oplus\YY]$, $S_A^\XX=S_A[\XB, \XX]$, and
$S_A^\YY=S_A[\YB,\YY]$. Given a greedy set $B$ of $f=(g,h)\in \XX\oplus\YY$, let
$A_1$, $A_2$ and $A_{12}$ be disjoint subsets of $\NN$ such that
\[
B=(A_{12}\cup A_1)^{\plus} \cup (A_{12}\cup A_2)^{\minus}.
\]
We have $|B|=2|A_{12}|+|A_1|+|A_2|$. Set $A_0=\NN\setminus (A_{12}\cup A_1\cup A_2)$. Let $(c_n)_{n=1}^\infty$ be the coefficients of $f$ relative to $\XB\diamond\YB$, let $(a_n)_{n=1}^\infty$ be the coefficients of $g$ relative to $\XB$, and let $(b_n)_{n=1}^\infty$ be the coefficients of $h$ with respect to $\YB$. If $c=\min\{|c_n| \colon n\in B\}$,
\[
\max_{n\in A_0} \left\{\frac{1}{\sqrt{2}} |a_n+b_n|, \frac{1}{\sqrt{2}} |a_n-b_n|\right\} =
\max_{n\in A_0^{\plus}\cup A_0^{\minus}} |c_n| \le c.
\]
Hence, $|a_n|$, $|b_n|\le \sqrt{2} c$ for all $n\in A_0$, i.e., $A_3\cup A_4\subseteq \NN\setminus A_0$, where
\[
A_3=\{n\in\NN \colon |a_n|>\sqrt{2} c\}, \quad A_4=\{n\in\NN \colon |b_n|>\sqrt{2} c\}.
\]
Note that $A_3$ is a greedy set of $g$, $A_4$ is a greedy set of $h$, and
\[
\max\{ |A_3|,|A_4|\}\le |\NN\setminus A_0|=|A_{12}\cup A_1\cup A_2|\le |A_{12}|+|A_1|+|A_2|\le |B|.
\]
Set $A_5=\NN\setminus (A_3\cup A_0)$ and $A_6=\NN\setminus (A_4\cup A_0)$.
Taking into account that, for any $D\subseteq\NN$, the coordinate projection on $\eta^{-1}(D)$ with respect to $\XB\diamond\YB$ coincides with that with respect to the direct sum $\XB\oplus\YB$ of bases $\XB$ and $\YB$ we obtain
\[
(S_{A_3}^\XX(g),S_{A_4}^\YY(h)) -S_B(f)=S_{A_1^{\minus}}(f) +S_{A_2^{\plus}}(f)-(S_{A_5}^\XX(g),S_{A_6}^\YY(h)).
\]
Therefore, it suffices to prove that
\[
\max\{ \Vert S_{A_5}^\XX(g)\Vert,
\Vert S_{A_6}^\YY(h)\Vert,
\Vert S_{A_1^{\minus}}(f) \Vert ,
\Vert S_{A_2^{\plus}}(f) \Vert\}
\le C_1 \Vert f \Vert
\]
for some constant $C_1$. Thus, the result would follow by applying the next two claims to the pairs of bases $(\XB,\YB)$, $(\XB,\YB^-)$, $(\YB,\XB)$, and $(\YB^-,\XB)$, where $\XB=(\xx_n)_{n=1}^\infty$, $\YB=(\yy_n)_{n=1}^\infty$, and $\YB^-=(-\yy_n)_{n=1}^\infty$.
\begin{claim}\label{claim1}
There is a constant $C$ such that
\[
\left\Vert \sum_{n\in A} a_n \, \xx_n\right\Vert \le C \left\Vert \sum_{n=1}^\infty a_n \, (\xx_n,\yy_n) + \sum_{n=1}^\infty b_n \, (\xx_n,-\yy_n) \right\Vert
\]
whenever $A\subseteq\NN$ is finite and $\max_{n\in A} |a_n|\le b:=\min_{n\in A} |b_n|$.
\end{claim}

\begin{claim}\label{claim2}
There is a constant $C$ such that
\[
\left\Vert \sum_{n\in A} a_n \, \xx_n\right\Vert \le C \left\Vert\left( \sum_{n=1}^\infty a_n \, \xx_n , \sum_{n=1}^\infty a_n \,\yy_n\right) \right\Vert
\]
whenever $\max_{n\in A} |a_n|\le b:=\min_{n\in A} \max\{|a_n+b_n|,|a_n-b_n|\}$.
\end{claim}

Let us prove Claim~\ref{claim1}. Set $D_1=\{n\in A \colon |a_n-b_n|\ge b\}$. If $n\in D_2:=A\setminus D_1$ then
\[
|a_n+b_n|=|2b_n+(a_n-b_n)| \ge 2 |b_n|-|a_n-b_n|> 2b-b=b.
\]
Hence, if $\kappa$ is the modulus of concavity of $\XX$, applying Lemma~\ref{lem:ShareBidem} we obtain
\begin{align*}
\left\Vert \sum_{n\in A} a_n \, \xx_n\right\Vert
&\le \kappa\left( \left\Vert \sum_{n\in D_1} a_n \, \xx_n\right\Vert +\left\Vert \sum_{n\in D_2} a_n \, \xx_n\right\Vert \right)\\
&\le \kappa C_0\left( \left\Vert \sum_{n=1}^\infty (a_n-b_n) \, \yy_n\right\Vert + \left\Vert \sum_{n=1}^\infty (a_n+b_n)\xx_n\right\Vert \right)\\
&\le 2 \kappa C_0 \max\left\{ \left\Vert \sum_{n=1}^\infty (a_n+b_n) \, \xx_n\right\Vert , \left\Vert \sum_{n=1}^\infty (a_n-b_n) \yy_n\right\Vert \right\}\\
&= 2 \kappa C_0\left\Vert \sum_{n=1}^\infty a_n \, (\xx_n,\yy_n) + \sum_{n=1}^\infty b_n \, (\xx_n,-\yy_n) \right\Vert.
\end{align*}

We conclude by proving Claim~\ref{claim2}. Set $D_1=\{n\in A \colon |a_n|\le |b_n| \}$ and $D_2=A\setminus D_1$. Since
\[
\max\{ |a_n|,|b_n|\}\ge \frac{b}{2}, \quad n\in A,
\]
we have $|b_n|\ge b/2$ for all $n\in D_1$ and $|a_n|\ge b/2$ for all $n\in D_2$. Therefore
\[
\max_{n\in D_1} |a_n|\le 2 \min_{n\in D_1} |b_n|, \quad
\max_{n\in D_2} |a_n|\le 2 \min_{n\in D_2} |a_n|.
\]
Applying Lemma~\ref{lem:ShareBidem} we obtain
\begin{align*}
\left\Vert \sum_{n\in A} a_n \, \xx_n\right\Vert
&\le \kappa \left(\left\Vert \sum_{n\in D_1} a_n \, \xx_n\right\Vert +\left\Vert \sum_{n\in D_2} a_n \, \xx_n\right\Vert \right)\nonumber\\
&\le \kappa C_0\left(\left\Vert \sum_{n =1}^\infty 2 b_n \, \yy_n\right\Vert +\left\Vert \sum_{n=1}^\infty 2 a_n \, \xx_n\right\Vert \right)\nonumber\\
&\le 4\kappa C_0 \left\Vert \left( \sum_{n =1}^\infty a_n \, \xx_n, \sum_{n=1}^\infty b_n \, \yy_n\right)\right\Vert.\qedhere
\end{align*}
\end{proof}

If $\varphi$ is the fundamental function of a basis of a Banach space, then $(\varphi(m))_{m=1}^\infty$ and $(m/\varphi(m))_{m=1}^\infty$ are non-decreasing sequences (see \cite{DKKT2003}). Our next result says that any such $\varphi$ corresponds in fact to a bidemocratic basis of a Banach space.

\begin{theorem}\label{thm:BDQGAFF}
Let $(s_m)_{m=1}^\infty$ be a non-decreasing unbounded sequence of positive scalars. Suppose that $(m/s_m)_{m=1}^\infty$ is unbounded and non-decreasing. Then there is a Banach space $\XX$ and a conditional bidemocratic quasi-greedy basis $\XB$ of $\XX$ whose fundamental function grows as $(s_m)_{m=1}^\infty$.
\end{theorem}

\begin{proof}
Let $\ww=(w_n)_{n=1}^\infty$ denote the weight whose primitive weight is $(s_m)_{m=1}^\infty$. Then $d_{1,1}(\ww)$ is a Banach space whose dual space is the Marcinkiewicz space $m(\ww)$, consisting of all $f\in c_0$ whose non-increasing rearrangement $(a_n)_{n=1}^\infty$ satisfies
\[
\Vert f\Vert_{m(\ww)}=\sup_m \frac{1}{s_m}\sum_{n=1}^m a_n<\infty
\]
(see \cite{CRS2007}*{Theorems 2.4.14 and 2.5.10}). Let $m_0(\ww)$ be the separable part of $m(\ww)$, and let
$\ww^{\ast}$ be the weight whose primitive weight is $(m/s_m)_{m=1}^\infty$.
We have the following chain of norm-one inclusions:
\begin{equation}\label{eq:embeddingsM}
d_{1,1}(\ww) \subseteq m_0(\ww^{\ast}) \subseteq m(\ww^{\ast}) \subseteq d_{1,\infty}(\ww).
\end{equation}
The right hand-side inclusion is clear. Let us prove the left hand-side inclusion. Let $(a_n)_{n=1}^\infty$ be the non-increasing rearrangement of $f\in c_0$. Given $m\in\NN$ we define $(b_n)_{n=1}^\infty$ by $b_n=a_n$ is $n\le m$ and $b_n=0$ otherwise. Using Abel's summation formula we obtain
\begin{align*}
\frac{s_m}{m}\sum_{n=1}^m a_n
&= \frac{s_m}{m} \sum_{n=1}^\infty (b_n-b_{n+1})n \\
&\le \sum_{n=1}^\infty (b_n-b_{n+1})s_n\\
&=\sum_{n=1}^m a_n w_n\le \Vert f\Vert_{1,\ww}.
\end{align*}

We infer from \eqref{eq:embeddingsM} that $d_{1,1}(\ww)$ and $m_0(\ww^{\ast})$ are Banach spaces for which the unit vector system is a symmetric basis with fundamental function $(s_m)_{m=1}^\infty$. Applying the rotation method with these bases yields a bidemocratic quasi-greedy basis of $d_{1,1}(\ww)\oplus m_0(\ww^{\ast})$ with fundamental function equivalent to $(s_m)_{m=1}^\infty$.

To show that this basis is conditional, by Proposition~\ref{prop:diamondConditional} it suffices to show that $d_{1,1}(\ww)$ and $m_0(\ww^{\ast})$ are not isomorphic, so that $d_{1,1}(\ww)\subsetneq m_0(\ww^{\ast})$. For that, we note that the unit vector system is a boundedly complete basis of $d_{1,1}(\ww)$ and that $\ell_1$ is a complemented subspace of $d_{1,1}(\ww^{\ast})$. Indeed, the first fact is clear. The second fact follows taking into account that the proof in \cite{ACL1973} works even without imposing to $\ww^*$ to be non-increasing. An appeal to \cite{AlbiacKalton2016}*{Theorem 3.2.15} and \cite{AlbiacKalton2016}*{Theorem 3.3.1} concludes the proof.
\end{proof}

\begin{remark}
Notice that in Theorem~\ref{thm:BDQGAFF} we can obtain that $\XB$ is $1$-bidemocratic with $\udf[\XB,\XX](m)=s_m$ for all $m\in\NN$. Indeed, if $(s_m)_{m=1}^\infty$ is a non-decreasing sequence of positive scalars such that $(m/s_m)_{m=1}^\infty$ is non-decreasing, and $\XX$ is a $p$-Banach space, $0<p\le 1$, with a bidemocratic basis $\XB$ such that $\udf[\XB,\XX](m)\approx s_m$ for $m\in\NN$, then, arguing as in the proof of \cite{DOSZ2011}*{Theorem 2.1} (where unconditionality plays no role), we obtain an equivalent $p$-norm for $\XX$ with respect to which $\udf[\XB,\XX](m)=s_m$ and $\udf[\XB^{\ast},\XX^{\ast}](m)=m/s_m$ for all $m\in\NN$.
\end{remark}

\begin{remark}
In the case when $(s_m)_{m=1}^\infty$ has the URP we can give a more quantitative approach to the proof of Theorem~\ref{thm:BDQGAFF}. In this particular case we have $m(\ww^{\ast})=d_{1,\infty}(\ww)$. Moreover, by \cite{CRS2007}*{Theorem 2.5.10}, $d_{1,q}(\ww)$ is a Banach space for every $1<q<\infty$. Notice that, in general, $d_{1,q}(\ww)$ is $r$-Banach for all $r<1$ and $1<q\le \infty$, and that $d_{1,q}(\ww)$ is $q$-Banach for all $0<q<1$. Applying the rotation method with the unit vector systems of $d_{1,p}(\ww)$ and $d_{1,q}(\ww)$, $0<p<q\le \infty$, yields a bidemocratic quasi-greedy basis (of a quasi-Banach space which is locally convex if $p\ge 1$) whose fundamental function is equivalent to $(s_m)_{m=1}^\infty$. Combining \eqref{eq:NormLorentz} with Proposition~\ref{prop:diamondConditional} gives that the conditionality constants $(\unc_m)_{m=1}^\infty$ of the basis we obtain satisfy
\[
\unc_m\gtrsim (H_m[\ww])^{1/p-1/q}, \quad m\in\NN.
\]
In the particular case that $(s_m)_{m=1}^\infty$ has the LRP, by Lemma~\ref{lem:LorentzLRP},
\[
\unc_m\gtrsim (\log m)^{1/p-1/q}, \quad m\in\NN,\; m\ge 2.
\]
Notice that, if $1<p<q<\infty$ and $(m/s_m^q)_{m=1}^\infty$ is non-decreasing, $\XX$ is superreflexive (see \cite{Altshuler1975}). In particular, we find a bidemocratic quasi-greedy basis of a Banach space with $\unc_m\gtrsim \log m$ for $m\ge 2$; and, for each $0<s<1$, a bidemocratic quasi-greedy basis of a superreflexive Banach space with $\unc_m\gtrsim (\log m)^s$ for $m\ge 2$. Thus, the rotation method serves to built `highly conditional' almost greedy bases (see \cite{AADK2019b} for background on this topic).
\end{remark}

\begin{example}
Let $\XX$ be a Banach space with a greedy, non-symmetric basis $\XB$ whose dual basis is also greedy. Then, if $\XB_\pi$ is a permutation of $\XB$ nonequivalent to $\XB$, we have that $\XB\diamond\XB_\pi$ is a conditional quasi-greedy basis of $\XX\oplus\XX$. For instance, in light of \cite{Temlyakov1998}*{Theorem 2.1}, this technique can be applied to the $L_p$-normalized Haar system to obtain a bidemocratic conditional quasi-greedy basis of $L_p([0,1])$, $p\in(1,2)\cup(2,\infty)$. Also, since, for the same values of $p$, the space $\ell_p$ has a greedy basis which is non-equivalent to the canonical basis (see \cite{DHK2006}*{Theorem 2.1}), this technique yields a bidemocratic conditional basis of $\ell_p$.
\end{example}



\end{document}